\documentclass[a4paper,11pt]{article}
\usepackage[top=3cm, bottom=4cm, left=3cm, right=3cm]{geometry}
\usepackage{t1enc} 
\usepackage[latin2]{inputenc}
\usepackage{lmodern}

\usepackage{amssymb}
\usepackage{amsmath}
\usepackage{amsthm}

\usepackage{graphicx}

\usepackage{tikz}
\usetikzlibrary{arrows,shapes,plotmarks}
\usepackage{caption}

\theoremstyle{plain}
\newtheorem{thm}{Theorem}
\newtheorem{lemma}[thm]{Lemma}

\newtheorem{claim}[thm]{Claim}

\theoremstyle{definition}
\newtheorem{defn}[thm]{Definition}
\newtheorem{conject}{Conjecture}

\theoremstyle{remark}
\newtheorem{remark}[thm]{Remark}

\DeclareMathOperator{\conv}{conv}

\DeclareMathOperator{\iconv}{iconv}

\begin{document}
 
\definecolor{mydarkgreen}{RGB}{0,100,0}
\tikzstyle{dashed} = [line width=0.4pt, dash pattern=on 3.0pt off 3.0pt]
\tikzstyle{bigred} = [mark=*,mark size=2.5pt,
  mark options={color=red}, color=red]
\tikzstyle{bigblue} = [mark=square*,mark size=2.5pt,
  mark options={color=blue}, color=blue]
\tikzstyle{bigblack} = [mark=x,mark size=3.5pt,
  mark options={color=black, line width=1pt}, color=black]
\tikzstyle{biggreen} = [mark=triangle*,mark size=4pt,
  mark options={color=mydarkgreen}, color=mydarkgreen]

\tikzstyle{sred} = [mark=*,mark size=2pt,mark options={color=red}, color=red]
\tikzstyle{sblue} = [mark=square*,mark size=2pt,
  mark options={color=blue}, color=blue]
\tikzstyle{sblack} = [mark=x,mark size=2pt,
  mark options={color=black, line width=1pt}, color=black]
\tikzstyle{sgreen} = [mark=triangle*,mark size=2.3pt,
  mark options={color=mydarkgreen}, color=mydarkgreen]
\tikzstyle{smarko} = [mark=o,mark size=2.5pt,
  mark options={color=black}, color=black]

\tikzstyle{nblue} = [rectangle, fill=blue, minimum size=5pt]
\tikzstyle{nred} = [circle, fill=red, minimum size=5pt]
\tikzstyle{ngreen} = [shape=regular polygon, regular polygon sides =3, 
  fill=mydarkgreen, minimum size=8pt]
\tikzstyle{nblack} = [cross out, line width=1.3pt,  
  draw=black, minimum size=3pt]
\tikzstyle{nempty} = [circle, line width=0.8pt,  
  draw=black, minimum size=6pt]
\tikzstyle{extended line}=[shorten >=-#1,shorten <=-#1]

\begin{center}
\textbf{\begin{LARGE}
5-colorable visibility graphs have bounded size or 4 collinear points
\end{LARGE}}\\
\textsc{Bálint Hujter, Sándor Kisfaludi-Bak}

\end{center}
\begin{abstract}
We investigate the question of finding a bound for the size of a
$\chi$-colorable finite visibility graph that have at most $\ell$ collinear
points.  This can be regarded as a relaxed version of the Big Line - Big Clique
\cite{KaraPorWood2005} conjecture.  We prove that any finite point set that has
at least 2311 points has either 4 collinear points or a visibility graph that
cannot be 5-colored. 
\end{abstract}

\section{Introduction}
Let $X$ be a finite set of points in the Euclidean plane. For a pair of points
$u,v\in X$ the open line segment with endpoints $u$ and $v$ will be denoted by
$(uv)$. The \textit{visibility graph} $G_X$ is a simple graph with vertex set
$X$, where the pair $u,v\in X$ is connected  if and only if $(uv)$ does not
contain any point from $X$. 

The starting point of our investigation is the following conjecture.

\begin{conject}[Big line, big clique \cite{KaraPorWood2005}]\label{blbc}
For any fixed $\ell$ and $k$ there is a constant $c=c(k,l)$ such that every
finite planar point set which has size at least $c$ has either $\ell$ collinear
points, or its visibility graph has a clique of size $k$. 
\end{conject}

This conjecture is currently open for all $k \ge 6$ and $l \ge 4$. Note that
the finiteness here is necessary, there is a counterexample if we allow
infinite point sets \cite{wood2010big}.

Let $mc_\ell(k)$ be the maximum cardinality of a finite set $X$ that has at most
$\ell$ collinear points and its visibility graph can be colored with at most
$k$ colors. If there is no such maximum, then let $mc_\ell(k)=\infty$. Note
that $mc_\ell(k) \le c(k+1,\ell+1)$, because maximum clique size is at least
the chromatic number. Based on this inequality the following weaker conjecture
can be formulated:

\begin{conject}[\cite{porwood2010}] \label{blbchi}
$mc_\ell (k)<\infty$ for all $k,\ell \ge 2$.
\end{conject}

The values of $mc_\ell(\le 3)$ have been established in a paper of Kára, Pór
and Wood in \cite{KaraPorWood2005}. The value of $mc_3(4)$ was found later by
Aloupis et al. \cite{aloupis2013blocking}, but proven in a slightly different
framework than ours in. They also showed some lower bounds for $mc_3(k)$.
The best known bounds for $k=4$ and general $\ell$ can be derived from the
theorem of Bar\'at et al. about empty pentagons \cite{barat2012empty}.

We summarized the progress on finding upper bounds for $mc_\ell(k)$ in the
below table. The new bound (our main result) is underlined.

\begin{tabular}{|c||c|c|c|c|c|}
\hline 
$mc_\ell (k)$ & $k=2$ & $k=3$ & $k=4$ & $k=5$ & $k\geq 6$ \\ 
\hline 
\hline
$\ell=2$ & $2$ & $3$ & $4$ & $5$ & $k$\\ 
\hline 
$\ell=3$ & $3$ & $6$ & $12$ &
\underline{$\leq 2310$}
& ? 
\\
\hline 
$\ell=4$ & $4$ & $6$ & 
$\leq 36$
& ? & ? \\
\hline 
$\ell \geq 5$ & $\ell$ & $\ell+2$ &
$\leq 328 \ell^2$
& ? & ? \\
\hline 
\end{tabular}  

\section{Blocking lemmas}

Let $X$ be any point set. We call $X$ \textit{properly colored} if any pair of
distinct points $x,y \in X$ are not visible to each other if they share the
same color.
 
A subset $U$ of $X$ in wich each point has color $c$ is called
\textit{c-empty} if every point in $X \cap \conv(U) \setminus U$ has different
color than $c$. $U$ is \textit{$k$-color-blocked by} a colored set $B \subset
\conv(U) \setminus U$ if $U \cup B$ is a properly colored set and $B$ has at
most $k$ colors.

A set of three non collinear points is called a \textit{triangle}.

\begin{defn}(Equivalence of two colored point sets)
Let $X$ and $Y$ be two arbitrary colored point sets on the plane.
We call $X$ and $Y$ \textit{equivalent} if there exists a bijection $\phi: X
\rightarrow Y$ satisfying the following conditions:
\begin{itemize}
    \item For any $x_0 \in X$ and any finite subset $\{ x_1, x_2, \dots, x_k
      \}$ of $X$:
      \[ x_0 \in \conv(\{x_1, \dots, x_k\}) \Longleftrightarrow \phi(x_0) \in
      \conv(\{\phi(x_1), \dots, \phi(x_k)\}) \]	
	\item For any $x',x'' \in X$, $x'$ and $x''$ have the same color in $X$ if and only 
	if $\phi(x)$ and $\phi(x'')$ have the same color in $Y$.
\end{itemize} 
\end{defn}

\begin{lemma}~\label{triangle-blocking}
A unicolored triangle cannot be $2$-color-blocked. 
\end{lemma}
\begin{proof}
Consider a minimal counterexample: a unicolored triangle $T$ blocked by a
colored set $B$, such that $B$ is minimal among all such counterexamples.
Assume that the color of $T$ is black and the colors of $B$ are blue and red. 

There must be a red or blue point on each side of $T$. As a consequence of
minimality, we have that $T$ is black-empty, since choosing the black point
inside which is closest to one of the sides of $T$ and the two endpoints of this
side would define a unicolored triangle that is $2$-color-blocked by less
points (it does not contain the blocking points on the two other sides of $T$).

If the three blocking points on the sides of $T$ have the same color then they
form a unicolored triangle that is 2-color-blocked by fewer points than $T$,
again contradicting minimality.

\noindent
\begin{minipage}{0.55\textwidth}
Hence we may assume (w.l.o.g.) that one of the points on the sides is red
($R_1$) and two are blue ($B_1,B_2$). Now $B_1$ and $B_2$ must be separated by
a red point ($R_2$), then $R_1$ and $R_2$ must be separated by a blue point
$B_3$. Finally $B_1, B_2$ and $B_3$ form a unicolored triangle that is
$2$-color-blocked by fewer points than $T$, contradiction.
\end{minipage}
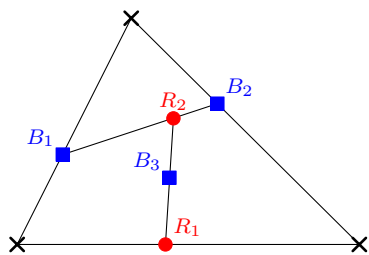
\begin{minipage}{0.4\textwidth}
\begin{center}
\begin{tikzpicture}[line cap=round,line join=round,x=0.75cm,y=0.75cm]
\clip(-0.68,-0.75) rectangle (7,4.86);
\draw (0,0)-- (2,4);
\draw (2,4)-- (6,0);
\draw (6,0)-- (0,0);
\draw (0.8,1.59)-- (3.51,2.49);
\draw (2.74,2.23)-- (2.6,0);
\begin{scriptsize}
\fill [bigblack] plot coordinates {(0,0)} node[above left] {};
\fill [bigblack] plot coordinates {(2,4)} node[above left] {};
\fill [bigblack] plot coordinates {(6,0)} node[above right] {};
\draw [bigblue]  plot coordinates {(0.8,1.59)} node[above left]{$B_1$};
\fill [bigblue] plot coordinates {(3.51,2.49)} node[above right] {$B_2$};
\fill [bigred] plot coordinates {(2.6,0)} node[above right] {$R_1$};
\fill [bigred] plot coordinates {(2.74,2.23)} node[above] {$R_2$};
\fill [bigblue] plot coordinates {(2.67,1.18)} node[above left] {$B_3$};
\end{scriptsize}
\end{tikzpicture}
\end{center}
\end{minipage}

\end{proof}

From this point onwards we fix $\ell=3$. It follows that in a properly coloured
set, the points of any color class are in general position.

\begin{claim}[\cite{KaraPorWood2005}]\label{mc3-3}
$mc_3(3) =6$.
\end{claim}
\begin{proof}
On one hand, if there were more than 6 points in a 3-colored point set, at
least three of them would have the same color.

\vspace{-0.1 cm}
\noindent
\begin{minipage}{0.67\textwidth}
Three unicolored points cannot be collinear (else at least two of them would be
adjacent in the visibility graph, as $\ell =3$), so they form a unicolored
triangle. Then this triangle would be 2-color-blocked, that contradicts
Lemma \ref{triangle-blocking}.

On the other hand, the figure on the right shows a properly 3-colored set with
6 points.
\end{minipage}
\begin{minipage}{0.3\textwidth}
\begin{center}
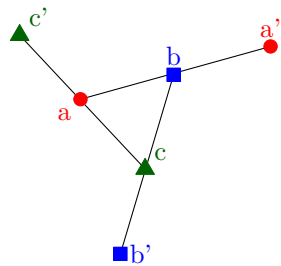

\begin{tikzpicture}[line cap=round,line join=round,>=triangle 45,
  x=2.5cm,y=2.5cm]
\begin{small}
\clip(0.5,-0.2) rectangle (2,1.5);
\draw (1.24,0.45)-- (0.58,1.16);
\draw (0.9,0.82)-- (1.9,1.1);(0,0)-- (1,2);
\draw (1.11,0)-- (1.39,0.95);(1,2)-- (3,0);
\draw [bigred] plot coordinates {(1.9,1.1)} node[above] {a'};
\fill [bigblue] plot coordinates{(1.11,0)} node[right] {b'};
\fill [biggreen] plot coordinates{(0.58,1.16)} node[above right] {c'};
\fill [biggreen] plot coordinates{(1.24,0.45)} node[above right] {c};
\draw [bigred] plot coordinates {(0.9,0.82)} node[below left] {a};
\fill [bigblue] plot coordinates{(1.39,0.95)} node[above] {b};
\end{small}
\end{tikzpicture}
\captionof{figure}{Figure $mc_3(3)$}
\end{center}
\end{minipage}

\end{proof}

\begin{claim}~\label{mc3-3.uniqueness}
Any properly 3-colored 6-point set is equivalent to the one shown on the figure
$mc_3(3)$.
\end{claim}
\begin{proof}
There must be exactly two points in every color class. Let $a$ and $a'$ be two
points from the first color class (red). These two must be blocked by a point
$b$ with different color (blue). $b$ has a pair $b'$, that has the same color.
$b$ and $b'$ cannot be blocked by either $a$ or $a'$, as it would mean that
$a,\ a',\ b$ and $b'$ would be on the same line. So $b$ and $b'$ are blocked by
a new point $c$ (from the third color class, green). $c$ has a pair $c'$ with
the same color. These two cannot be blocked by $b$ or $b'$, that would cause 4
points on the same line. So they are blocked by $a$ or $a'$. Since $a$ and $a'$
played a symmetric role so far, we may assume w.l.o.g. that the blocker is $a$.
Now our notations give the desired bijection to Figure $mc_3(3)$.
\end{proof}

For any point set $X$, we denote by $\iconv(X)$ the interior of $\conv(X)$.

\noindent
\begin{minipage}{0.74\textwidth}
\begin{lemma}
Suppose that $\ell =3$ and consider a unicolored, color-empty triangle $T$. If
$T$ is 3-color-blocked by a set $B\subset \conv T$, then $B \cup T$ is equivalent to
one of the five instances below:
\end{lemma}
\end{minipage}
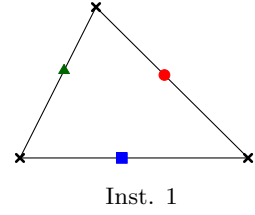
\begin{minipage}{0.24\textwidth}
\begin{tikzpicture}[line cap=round,line join=round,>=triangle 45
  ,x=1.0cm,y=1.0cm]
\clip(-0.2,-0.3) rectangle (3.1,2.5);
\draw (0,0)-- (3,0);
\draw (0,0)-- (1,2);
\draw (1,2)-- (3,0);
\fill [sblack] plot coordinates{(0,0)} node[below left] {};
\fill [sblack] plot coordinates{(3,0)} node[below right] {};
\fill [sblack] plot coordinates{(1,2)} node[above right] {};
\draw [sred] plot coordinates {(1.9,1.1)} node[above left] {};
\fill [sblue] plot coordinates{(1.34,0)} node[above right] {};
\fill [sgreen] plot coordinates {(0.58,1.16)} node[below left] {};
\end{tikzpicture}
\vspace{-0.4cm}
\captionof{figure}{Inst. 1}
\end{minipage}

\noindent
\begin{minipage}{0.24\textwidth}
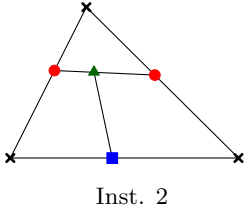
\begin{tikzpicture}[line cap=round,line join=round,>=triangle 45
  ,x=1.0cm,y=1.0cm]
\clip(-0.2,-0.3) rectangle (3.1,2.5);
\draw (0,0)-- (3,0);
\draw (0,0)-- (1,2);
\draw (1,2)-- (3,0);
\draw (1.9,1.1)-- (0.58,1.16);
\draw (1.1,1.14)-- (1.34,0);
\fill [sblack] plot coordinates{(0,0)} node[below left] {};
\fill [sblack] plot coordinates{(3,0)} node[below right] {};
\fill [sblack] plot coordinates{(1,2)} node[above right] {};
\draw [sred] plot coordinates {(1.9,1.1)} node[above left] {};
\fill [sblue] plot coordinates{(1.34,0)} node[above right] {};
\draw [sred] plot coordinates {(0.58,1.16)} node[below left] {};
\fill [sgreen] plot coordinates {(1.1,1.14)} node[below left] {};
\end{tikzpicture}
\vspace{-0.4cm}
\captionof{figure}{Inst. 2}
\end{minipage}
\begin{minipage}{0.24\textwidth}
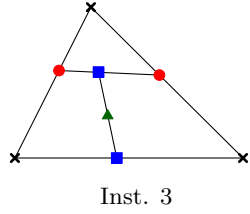
\begin{tikzpicture}[line cap=round,line join=round,>=triangle 45
  ,x=1.0cm,y=1.0cm]
\clip(-0.2,-0.3) rectangle (3.1,2.5);
\draw (0,0)-- (3,0);
\draw (0,0)-- (1,2);
\draw (1,2)-- (3,0);
\draw (1.9,1.1)-- (0.58,1.16);
\draw (1.1,1.14)-- (1.34,0);
\fill [sblack] plot coordinates{(0,0)} node[below left] {};
\fill [sblack] plot coordinates{(3,0)} node[below right] {};
\fill [sblack] plot coordinates{(1,2)} node[above right] {};
\draw [sred] plot coordinates {(1.9,1.1)} node[above left] {};
\fill [sblue] plot coordinates{(1.34,0)} node[above right] {};
\draw [sred] plot coordinates {(0.58,1.16)} node[below left] {};
\fill [sblue] plot coordinates {(1.1,1.14)} node[below left] {};
\fill [sgreen] plot coordinates {(1.22,0.56)} node[below left] {};
\end{tikzpicture}
\vspace{-0.4cm}
\captionof{figure}{Inst. 3}
\end{minipage}
\begin{minipage}{0.24\textwidth}
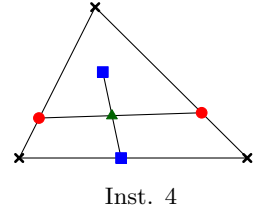
\begin{tikzpicture}[line cap=round,line join=round,>=triangle 45
  ,x=1.0cm,y=1.0cm]
\clip(-0.2,-0.3) rectangle (3.1,2.5);
\draw (0,0)-- (3,0);
\draw (0,0)-- (1,2);
\draw (1,2)-- (3,0);
\draw (0.26,0.53)-- (2.4,0.6);
\draw (1.1,1.14)-- (1.34,0);
\fill [sblack] plot coordinates{(0,0)} node[below left] {};
\fill [sblack] plot coordinates{(3,0)} node[below right] {};
\fill [sblack] plot coordinates{(1,2)} node[above right] {};
\draw [sred] plot coordinates {(0.26,0.53)} node[above left] {};
\fill [sblue] plot coordinates{(1.34,0)} node[above right] {};
\draw [sred] plot coordinates {(2.4,0.6)} node[below left] {};
\fill [sblue] plot coordinates {(1.1,1.14)} node[below left] {};
\fill [sgreen] plot coordinates {(1.22,0.56)} node[below left] {};
\end{tikzpicture}
\vspace{-0.4cm}
\captionof{figure}{Inst. 4}
\end{minipage}
\begin{minipage}{0.24\textwidth}
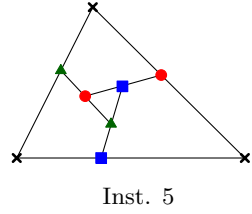
\begin{tikzpicture}[line cap=round,line join=round,>=triangle 45
  ,x=1.0cm,y=1.0cm]
\clip(-0.2,-0.3) rectangle (3.1,2.5);
\draw (0,0)-- (3,0);
\draw (0,0)-- (1,2);
\draw (1,2)-- (3,0);
\draw (1.24,0.45)-- (0.58,1.16);
\draw (0.9,0.82)-- (1.9,1.1);(0,0)-- (1,2);
\draw (1.11,0)-- (1.39,0.95);(1,2)-- (3,0);
\fill [sblack] plot coordinates{(0,0)} node[below left] {};
\fill [sblack] plot coordinates{(3,0)} node[below right] {};
\fill [sblack] plot coordinates{(1,2)} node[above right] {};
\draw [sred] plot coordinates {(1.9,1.1)} node[above left] {};
\fill [sblue] plot coordinates{(1.11,0)} node[above right] {};
\fill [sgreen] plot coordinates{(0.58,1.16)} node[above right] {};
\fill [sgreen] plot coordinates{(1.24,0.45)} node[above right] {};
\draw [sred] plot coordinates {(0.9,0.82)} node[above left] {};
\fill [sblue] plot coordinates{(1.39,0.95)} node[above right] {};
\end{tikzpicture}
\vspace{-0.4cm}
\captionof{figure}{Inst. 5}
\end{minipage}

\begin{proof}
$B$ must have one point on each side of $T$, we denote the set of these three
points by $B^s$, and $B^{in} = B \setminus B^s$. $B$ is a properly 3-colored
set, so $|B| \leq 6$ by Claim \ref{mc3-3}. Hence $|B^{in}| \leq 3$.
 
\noindent \textbf{Case 1:} $|B^{in}|=0$ \\ \noindent The points of $B^s$ see
each other, so they must have different colors.  It means that $B \cup T$ is
equivalent to Instance 1. 

\noindent \textbf{Case 2:} $|B^{in}|=1$ \\
\noindent The only point $p$ of $B$ in $\iconv(T)$ can block only one pair of
$B^s$, so $B^s$ needs at least 2 colors. $p$ sees all the other points of $T
\cup B$, so it must have a unique color. Then only two colors remain for $B^s$,
the only way to color it with two colors is shown in Instance 2.  

\noindent \textbf{Case 3:} $|B^{in}|=2$ \\
\noindent The two points of $B^{in}$ see each other, so they have different
colors. They cannot block all three visibilities between $B^s$, so $B^{in}$
must have at least two colors, too. Hence there are points $i \in B^{in}$ and
$s_1 \in B^s$ with the same color. $i$ and $s_1$ can only be blocked by the
other point of $B^{in}$, call it $i'$. Now $i$ and $i'$ has different colors,
and both can see both points of $B^s \setminus \{s_1\} = \{s_2,s_3\}$, hence
$s_2$ and $s_3$ must have the same color.  So the visibility between $s_2$ and
$s_3$ has to be blocked. The blocking point can be either $i$ or $i'$, those
correspond to the cases (3) and (4).  

\noindent \textbf{Case 4:} $|B^{in}|=3$ \\
\noindent Now $|B|=6$ and $B$ is properly 3-colored, so by Claim
\ref{mc3-3.uniqueness}, $B$ is equivalent to the set shown on Figure $mc_3(3)$.
It is easy to check that $B^s$ must be formed by $a',\ b'$ and $c'$ (unless
some of them would be outside $T$), then the equivalence is straightforward.

\end{proof}

\smallskip

The following 2 lemmas and Theorem \ref{mc_3(4)} are established in
\cite{aloupis2013blocking}.  Despite the differences between the definitions, their
proofs are directly applicable here. We included the proofs in the appendix for
completeness.

\noindent
\begin{minipage}{0.5\textwidth}
\begin{lemma}[\cite{aloupis2013blocking}] \label{convex4block}
Let $Q$ be the set of the vertices of a convex quadrilateral. Suppose that $Q$
is unicolored and color-empty. Then any blocking set $B$ of $Q$ is equivalent
to the one shown on the figure at the right.
\end{lemma}
\end{minipage}
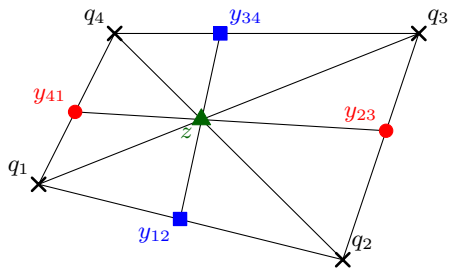
\begin{minipage}{0.4\textwidth}
\begin{flushright}
\begin{tikzpicture}[line cap=round,line join=round,x=1.0cm,y=1.0cm]
\clip(-0.5,-1.1) rectangle (5.4,3);
\draw (0,0)-- (4,-1);
\draw (4,-1)-- (5,2);
\draw (5,2)-- (1,2);
\draw (1,2)-- (0,0);
\draw (1,2)-- (4,-1);
\draw (0,0)-- (5,2);
\draw (0.48,0.96)-- (4.57,0.71);
\draw (2.39,2)-- (1.86,-0.46);
\begin{footnotesize}
\fill [bigblack] plot coordinates{(0,0)} node[above left] {$q_1$};
\fill [bigblack] plot coordinates {(4,-1)} node[above right] {$q_2$};
\fill [bigblack] plot coordinates {(5,2)} node[above right] {$q_3$};
\fill [bigblack] plot coordinates {(1,2)} node[above left] {$q_4$};
\fill [biggreen] plot coordinates {(2.14,0.86)} node[below left] {$z$};
\fill [bigblue] plot coordinates {(1.86,-0.46)} node[below left] {$y_{12}$};
\draw [bigred] plot coordinates {(4.57,0.71)} node[above left] {$y_{23}$};
\fill [bigblue] plot coordinates {(2.39,2)} node[above right] {$y_{34}$};
\draw [bigred] plot coordinates {(0.48,0.96)} node[above left] {$y_{41}$};
\end{footnotesize}
\end{tikzpicture}
\end{flushright}
\end{minipage}

\begin{remark}[\cite{aloupis2013blocking}]
The above 9-point set is maximal, i.e. if a 4-colored visibility
graph has 4 unicolored points in convex position, then it has exactly 9 points.
\end{remark}

\begin{lemma}[\cite{aloupis2013blocking}] \label{concave4block}
A unicolored concave set of 4 points can be blocked by 3 colors only the
following way:

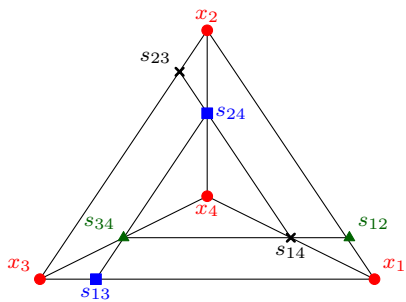
\begin{center}
\begin{tikzpicture}[line cap=round,line join=round,x=1.1cm,y=1.1cm]
\clip(-2.5,-1.28) rectangle (2.46,2.4);
\draw (2,-1)-- (0,2);
\draw (-2,-1)-- (2,-1);
\draw (0,2)-- (-2,-1);
\draw (0,0)-- (0,2);
\draw (0,0)-- (2,-1);
\draw (0,0)-- (-2,-1);
\draw (-1,-0.5) -- (1.7,-0.5);
\draw (-0.33,1.5) -- (1,-0.5);
\draw (-1.33,-1) -- (0,1);
\begin{scriptsize}
\fill [sred] plot coordinates{(2,-1)} node[above right] {$x_1$};
\fill [sred] plot coordinates{(0,2)} node[above] {$x_2$};
\fill [sred] plot coordinates{(-2,-1)} node[above left] {$x_3$};
\fill [sred] plot coordinates{(0,0)} node[below] {$x_4$};
\fill [sblue] plot coordinates{(0,1)} node[right] {$s_{24}$};
\draw [sblack] plot coordinates{(1,-0.5)} node[below] {$s_{14}$};
\fill [sgreen] plot coordinates{(1.7,-0.5)} node[above right] {$s_{12}$};

\fill [sblue] plot coordinates{(-1.33,-1)} node[below] {$s_{13}$};
\fill [sgreen] plot coordinates{(-1,-0.5)} node[above left] {$s_{34}$};
\draw [sblack] plot coordinates{(-0.33,1.5)} node[above left] {$s_{23}$};
\end{scriptsize}
\end{tikzpicture}
\end{center}
\end{lemma}

\begin{remark}[\cite{aloupis2013blocking}]
The above 10-point set is maximal, i.e. if a 4-colored visibility
graph has 4 unicolored points in concave position, then it has exactly 10
points.
\end{remark}

\begin{thm} \label{mc_3(4)}
$mc_3(4)=12$.
\end{thm}

\section{Main result}

\begin{thm}\label{main}
  A unicolored convex hexagon cannot be 4-color-blocked if $\ell=3$.
\end{thm}

\begin{proof}
In a unicolored convex hexagon the blocking set $X$ has to block 15 segments
defined by the 6 vertices and endpoints. We denote the vertex set of the
hexagon by $H$. The 6 edges of the hexagon need distinct blocking points. All
other points may block at most two diagonals, except maybe one blocker: if the
diagonals connecting opposite vertices are concurrent, then these can be
blocked by one point. It follows that the number of blocking points needed is
at least 10 (6 for the edges, and at least 4 for the 9 diagonals).

By Theorem \ref{mc_3(4)}, it follows that a blocking set cannot have more than
12 points.  Thus it is enough to show that $H$ cannot be 4-color-blocked by
10, 11 or 12 points.

We may assume that our blocking set $X$ lies in $\conv(H)$.
It is easy to observe that $X$ has at least 6 points in $\partial \conv(X)$:
the ones that block the edges of the hexagon.

To prove this theorem, we will use the following simple lemmas.

\begin{lemma}\label{atmost3}
The biggest color class of $X$ has size at most 3.
\end{lemma}
\begin{proof}
Suppose that a color class $C$ has at least 4 points. By Lemmas
\ref{convex4block} and \ref{concave4block}, we get that $X$ contains
one of two configurations. Since both of these configurations are maximal, $X$
is also equivalent to one of these configurations. We observe that for both
these configurations the number of points on the boundary of the convex hull is
less than 6, thus the equivalent of these configurations cannot block the
hexagon. 
\end{proof}

It follows that there are at least two unicolored triangles in the blocking
set.

\begin{lemma}\label{notonboundary}
A unicolored triangle $T$ of color $c$ cannot have all its vertices on the
edges of the hexagon.
\end{lemma}
\begin{proof}
Take any unicolored triangle $T'\neq T$ of color $c'$. To block $T'$, there
must be points of each of the $3$ blocking colors that differ from $c'$. Thus
$T'$ must contain  a point of color $c$. The only way this is possible inside
the convex hull of the hexagon is if two vertices of $T'$ lie on the same edge
as a point of color $c$. But then there would be 5 points on this edge: we
arrived at a contradiction.
\end{proof}

We are now ready to prove the theorem. We will check the cases
$|X|=10,\; 11, \; 12$.

\smallskip
\noindent\textbf{Case 1: } $|X|=10$. Take any unicolored triangle $T$. At least
one of its vertices must be in $\iconv(H)$ by Lemma \ref{notonboundary}. We
need at least $3$ points to block $T$, and all of these points lie in
$\iconv(H)$. It follows that all 4 points in $\iconv(H)$ lie on the boundary of
$T$, and they are in convex position. On the other hand, it can be verified
that the inner points of a 10-point blocking set of a convex hexagon need to be
in concave position (the intersection $p$ of the 3 diagonals that connect
opposite vertices will be in the convex hull of the three other inner blocking
points).

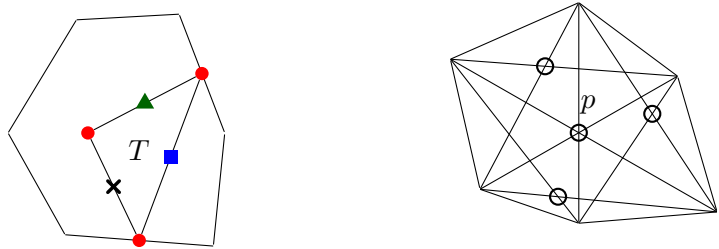
\begin{center}
\begin{tikzpicture}[line cap=round,line join=round,>=triangle 45
  ,x=1.5cm,y=1.5cm,extended line/.style={shorten >=-#1,shorten <=-#1}]
\clip(-1,0) rectangle (3,2.2);
\tikzstyle{every node}=[inner sep=0pt, minimum width=0pt]
\node (A) at (0,0.1) {};
\node (B) at (1.3,0) {};
\node (C) at (1.4,1) {};
\node (D) at (1,2.05) {};
\node (E) at (0.1,2) {};
\node (F) at (-0.5,1) {};
\draw (A)-- (B) --(C) -- (D) --(E) --(F) --(A);
\node [nred] (P) at (0.2,1) {};
\node [nred] (AB) at (0.65,0.05) {};
\node [nred] (CD) at (1.2,1.525) {};
\node [node distance = 1.2 cm, above of = AB, color=black] {$T$};
\draw (AB) edge node[nblue](x){} (CD) (CD) edge node[ngreen](y){}
   (P) (P) edge node[nblack](z){} (AB);
\end{tikzpicture}
\hspace{0.2cm}
\begin{tikzpicture}[line cap=round,line join=round,>=triangle 45,
  x=1.5cm,y=1.5cm]
\clip(-1.2,-1) rectangle (1.5,1.5);
\tikzstyle{every node}=[inner sep=0pt, minimum width=0pt]
\node (A) at (30:1) {};
\node (B) at (90:1.15) {};
\node (C) at (150:1.3) {};
\node (D) at (-150:1) {};
\node (E) at (-90:0.8) {};
\node (F) at (-30:1.4) {};
\node [nempty,label={[above right=3.5 pt] $p$}] at (0,0){};
\node [nempty] at (intersection of A--C and B--D){};
\node [nempty] at (intersection of E--C and F--D){};
\node [nempty] at (intersection of A--E and B--F){};
\draw (A)-- (B) --(C) -- (D) --(E) --(F) --(A);
\draw (A)--(C)--(E)--(A);
\draw (B)--(D)--(F)--(B);
\draw (A)--(D);
\draw (B)--(E);
\draw (C)--(F);
\end{tikzpicture}
\end{center}

\noindent\textbf{Case 2: } $|X|=11$. It follows that there are at least two
points in $\iconv(H)$ that have the same color (red), we denote them by $p$ and
$q$. We distinguish two subcases based on the number of red points.

(2a) There are two red points: $p$ and $q$. It follows that all other classes
have 3 points. Let $v$ be the (blue) blocking point of $p$ and $q$. The blue
triangle $T=(vv'v'')$ can not have any red point on $vv'$ or $vv''$, because
then there would be 4 collinear points. It also can not have any red point in
$\iconv(T)$, because then there would be $5$ points in $\iconv(H)\cap
\conv(T)$, and $\iconv(H)$ has one more point outside $\conv(T)$ since either
$p$ or $q$ is there. It follows that there is a red point on $v'v''$, suppose
it is $p$.

The green blocking point $g$ on the segment $vv''$ can see all points on the
right side of $\overrightarrow{pq}$, thus the remaining green points must lie
on the left side. The only point that can block them is the black point $b$.
Consequently $p$ and $q$ lie on the sides of the green triangle. Similarly, $p$
and $q$ lie on the sides of the black triangle.  It follows that there is a
black and a green blocking point on the right side of $\overrightarrow{v'v''}$,
and these points are on the sides of $H$. So $v'$ and $v''$ lie on the opposite
sides of the hexagon. Since all inner points are on the left side of
$\overrightarrow{v'v''}$, we get that there is a diagonal $d$ that is not
blocked.

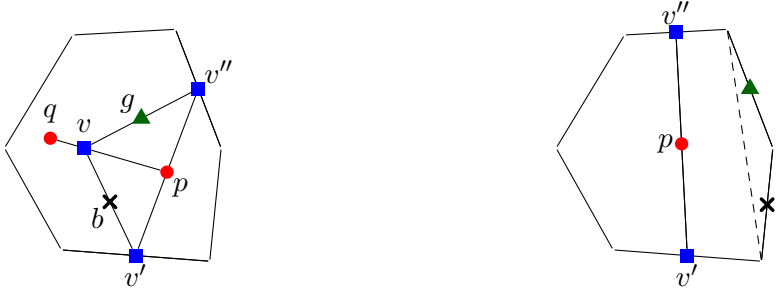
\begin{center}
\begin{tikzpicture}[line cap=round,line join=round,x=1.5cm,y=1.5cm,
label distance = 3 pt]
\clip(-1,-0.5) rectangle (3,2.2);
\tikzstyle{every node}=[inner sep=0pt]
\node (A) at (0,0.1) {};
\node (B) at (1.3,0) {};
\node (C) at (1.4,1) {};
\node (D) at (1,2.05) {};
\node (E) at (0.1,2) {};
\node (F) at (-0.5,1) {};
\draw (A)-- (B) --(C) -- (D) --(E) --(F) --(A);
\node [nblue, label={above: $v$}] (v) at (0.2,1) {};
\draw (A) -- node[nblue, label={below: $v'$}](AB){} (B)
      (C) -- node[nblue, label={above right:$v''$}] (CD){} (D);
\draw (AB) -- node[nred, label={below right: $p$}](p){} (CD)
      (CD) -- node[ngreen, label={above left: $g$}](y){} (v)
      (v) -- node[nblack, label={below left: $b$}](z){} (AB);
\draw (p) -- (v) node[nred,pos=1.6,label={above: $q$}](q){};
\draw (v) -- (q);
\end{tikzpicture}
\hspace{1cm}
\begin{tikzpicture}[line cap=round,line join=round,x=1.5cm,y=1.5cm,
label distance = 3 pt]
\clip(-1,-0.5) rectangle (3,2.4);
\tikzstyle{every node}=[inner sep=0pt]
\node (A) at (0,0.1) {};
\node (B) at (1.3,0) {};
\node (C) at (1.4,1) {};
\node (D) at (1,2.05) {};
\node (E) at (0.1,2) {};
\node (F) at (-0.5,1) {};
\draw (A) -- node [nblue, label={below: $v'$}](AB){} (B);
\draw (B)--(C)-- (D);
\draw (D) --  node [nblue, label={above:$v''$}](DE){} (E);
\draw (E) --(F) --(A);
\draw (AB)--(DE);
\draw [dashed] (B)--(D);
\draw (B) -- node [nblack](BC){} (C);
\draw (C) -- node [ngreen](CD){} (D);
\draw (AB) -- node [nred, label={left: $p$}](p){} (DE);
\end{tikzpicture}
\end{center}

(2b) There are 3 red points. It follows that all inner points are on the convex
hull of the red triangle (because the red triangle needs at least 3 more inner
points to be blocked). Let $v$ be the blue point that blocks the two inner red
points $p$ and $q$.  The other inner points need to be black ($b$) and green
($g$). Wlog.\ we can suppose $v$ is on the left side of $\overrightarrow{bg}$.

There cannot be $3$ blue points, because only $b$ and $g$ can block points from
$v$, but then there would be two blue points on the right side of
$\overrightarrow{bg}$.  Since all inner points are on the left side (or on line
$bg$ itself), these would see each other. Thus there are $2$ blue points, $3$
black points and $3$ green points.

The remaining 2 black points could be blocked from $b$ by $v$,$p$, or $g$. If a
black point lies on the line $bp$, then it will see every possible for the
third black point on $bg$ or $bv$, thus no black point lies on $bp$. Similarly,
there is no further green point on $gq$.  It follows that at least one of the
dashed lines will have more than one of the remaining 5 points, which would
mean 4 points on one line.

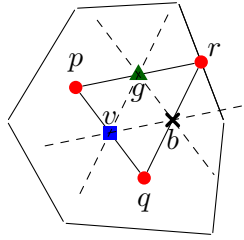
\begin{center}
\begin{tikzpicture}[line cap=round,line join=round,x=1.5cm,y=1.5cm,
label distance = 3 pt]
\clip(-1,-0.5) rectangle (3,2.2);
\tikzstyle{every node}=[inner sep=0pt]
\node (A) at (0,0.1) {};
\node (B) at (1.3,0) {};
\node (C) at (1.4,1) {};
\node (D) at (1,2.05) {};
\node (E) at (0.1,2) {};
\node (F) at (-0.5,1) {};
\draw (A)-- (B) --(C) -- (D) --(E) --(F) --(A);
\node [nred, label={above: $p$}] (p) at (0.1,1.3) {};
\node [nred, label={below: $q$}] (q) at (0.7,0.5) {};
\draw (C) -- node[nred, label={above right:$r$}] (CD){} (D);
\draw (CD)-- node[ngreen, label={below: $g$}](g){} (p)
      (p) -- node[nblue, label={above: $v$}](b){} (q)
      (q) -- node[nblack, label={below: $b$}](v){} (CD);
\draw [extended line = 1cm, dashed] (b)--(g); 
\draw [extended line = 1cm, dashed] (g)--(v);
\draw [extended line = 1cm, dashed] (v)--(b);
\end{tikzpicture}
\end{center}

\noindent\textbf{Case 3: } $|X|$=12. Each color class has $3$ points. By
earlier observations, no color class can lie exclusively on the sides of $H$.
It follows that the size of color classes in $\iconv(H)$ is either $3,1,1,1$ or
$2,2,1,1$.

(3a) There is a unicolored triangle in $\iconv(H)$. This case is very similar
to case (2b). A blue point can not be blocked from $v$ by $r$, because such a
point would see all other possible third blue points on line $vb$ and $vq$, and
the same can be said for the pair $b,p$ and $q,g$. So the remaining 6 points
would need to lie on the three dashed lines, resulting in at least 4 points on
one line.

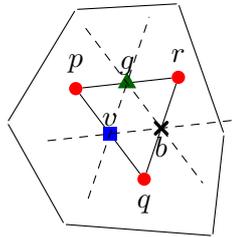
\begin{center}
\begin{tikzpicture}[line cap=round,line join=round,x=1.5cm,y=1.5cm,
label distance = 3 pt]
\clip(-1,-0.5) rectangle (3,2.2);
\tikzstyle{every node}=[inner sep=0pt]
\node (A) at (0,0.1) {};
\node (B) at (1.3,0) {};
\node (C) at (1.4,0.9) {};
\node (D) at (1,2.05) {};
\node (E) at (0.1,2) {};
\node (F) at (-0.5,1) {};
\draw (A)-- (B) --(C) -- (D) --(E) --(F) --(A);
\node [nred, label={above: $p$}] (p) at (0.1,1.3) {};
\node [nred, label={below: $q$}] (q) at (0.7,0.5) {};
\node [nred, label={above: $r$}] (r) at (1.0,1.4){};
\draw (r)-- node[ngreen, label={above: $g$}](g){} (p)
      (p) -- node[nblue, label={above: $v$}](b){} (q)
      (q) -- node[nblack, label={below: $b$}](v){} (r);
\draw [extended line = 1cm, dashed] (b)--(g); 
\draw [extended line = 1cm, dashed] (g)--(v);
\draw [extended line = 1cm, dashed] (v)--(b);
\end{tikzpicture}
\end{center}

(3b) The size of color classes in $\iconv(H)$ is $2,2,1,1$. Suppose there are 2
red and 2 blue points. At least one of the inner red and blue point pairs is
blocked by a green or black point, wlog.\ we can suppose the two red points
$r_1$ and $r_2$ are blocked by the green point $g$. We denote by $G$ the
vertices of the green triangle. There
must be at least one red point ($r_1$) in $\conv(G) \setminus G$. Consequently $r_2$ lies
outside $\conv(G)$.

First we consider the case where $r_1$ lies in $\iconv(G)$. There must be 4
points in $\conv(G) \setminus G$. It follows that $r_1$ is the blocker of the
two inner blue points $v_1$ and $v_2$, and the inner black point $b$ lies on
the third side of the triangle determined by $G$.

If the beam $v_1v_2$ lies on the neighbouring sides of $g$, then the third blue
point $v$ can not be blocked from both $v_1$ and $v_2$. To see this we can
suppose wlog. that $v$ is on the same side of $\overrightarrow{r_1r_2}$ as
$v_1$. But then there are no points that could block $v$ and $v_1$.

If the beam $v_1v_2$ is positioned otherwise, then the only point that can
block the third red point $r$ from $r_1$ is $b$. It follows that $r$ and $r_2$
can see each other.

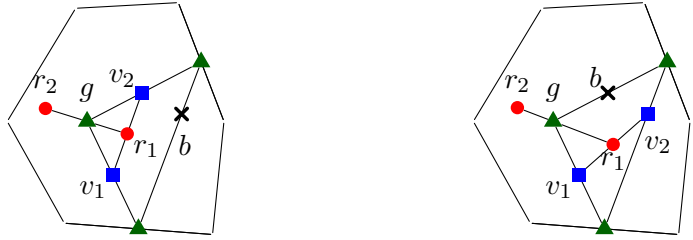
\begin{center}
\begin{tikzpicture}[line cap=round,line join=round,x=1.5cm,y=1.5cm,
label distance = 3 pt]
\clip(-1,-0.5) rectangle (3,2.2);
\tikzstyle{every node}=[inner sep=0pt]
\node (A) at (0,0.1) {};
\node (B) at (1.3,0) {};
\node (C) at (1.4,1) {};
\node (D) at (1,2.05) {};
\node (E) at (0.1,2) {};
\node (F) at (-0.5,1) {};
\draw (A)-- (B) --(C) -- (D) --(E) --(F) --(A);
\node [ngreen, label={above: $g$}] (g) at (0.2,1) {};
\draw (A) -- node[ngreen](AB){} (B)
      (C) -- node[ngreen] (CD){} (D);
\draw (AB) -- node[nblack, pos=0.7, label={right: $b$}](b){} (CD)
      (CD) -- node[nblue, label={above left: $v_2$}](v2){} (g)
      (g) -- node[nblue, label={below left: $v_1$}](v1){} (AB)
      (v1) -- node[nred, label={below right: $r_1$}](r1){} (v2);
\draw (r1) -- (g) node[nred,pos=3.0,label={above: $r_2$}](r2){};
\draw (g) -- (r2);
\end{tikzpicture}
\begin{tikzpicture}[line cap=round,line join=round,x=1.5cm,y=1.5cm,
label distance = 3 pt]
\clip(-1,-0.5) rectangle (3,2.2);
\tikzstyle{every node}=[inner sep=0pt]
\node (A) at (0,0.1) {};
\node (B) at (1.3,0) {};
\node (C) at (1.4,1) {};
\node (D) at (1,2.05) {};
\node (E) at (0.1,2) {};
\node (F) at (-0.5,1) {};
\draw (A)-- (B) --(C) -- (D) --(E) --(F) --(A);
\node [ngreen, label={above: $g$}] (g) at (0.2,1) {};
\draw (A) -- node[ngreen](AB){} (B)
      (C) -- node[ngreen] (CD){} (D);
\draw (AB) -- node[nblue, pos=0.7, label={right: $v_2$}](v2){} (CD)
      (CD) -- node[nblack, label={above left: $b$}](b){} (g)
      (g) -- node[nblue, label={below left: $v_1$}](v1){} (AB)
      (v1) -- node[nred, label={below: $r_1$}](r1){} (v2);
\draw (r1) -- (g) node[nred,pos=2.0,label={above: $r_2$}](r2){};
\draw (g) -- (r2);
\end{tikzpicture}
\end{center}

Now we consider the case where $r_1$ lies on a side of $GT$. The red triangle
$RT=(r_1r_2r)$ lies in a closed half plane determined by the line $r_1r_2$, and
a side $s$ of $GT$ necessarily lies outside this half plane.
The blocking point of $s$ lies outside $\conv(RT)$, so $RT$ must be
blocked by 3 points of different colors: $p,q$ and $g$. Note that one of these
three blocking points is also needed to block a side of $GT$, suppose this
point is $p$.

$RT$ will have a black point $b$ and an inner blue point (suppose $v_1$) on its
sides. (So $\{p,q\}=\{b,v_1\}$.) Let $v_2$ be the other point that is outside
$RT$ and blocks a side of $GT$. Since there are 2 blue points in $\iconv(H)$,
$v_2$ must be blue.

We can distinguish two cases depending on the role of $p$: it can either block
$r_1r$ or $r_2r$. It is easy to verify that in both cases $v_2$ can see both
$p$ and $q$, but one of them is blue, which concludes the proof of this
theorem.  \qedhere

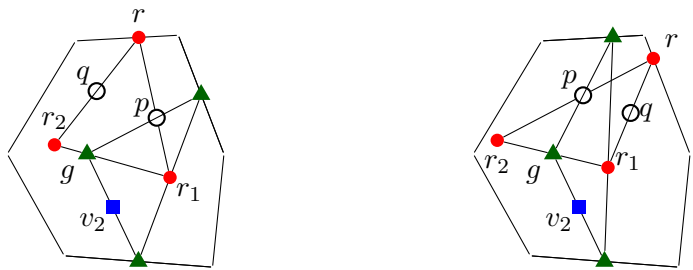
\begin{center}
\begin{tikzpicture}[line cap=round,line join=round,x=1.5cm,y=1.5cm,
label distance = 3 pt]
\clip(-1,-0.5) rectangle (3,2.4);
\tikzstyle{every node}=[inner sep=0pt]
\node (A) at (0,0.1) {};
\node (B) at (1.3,0) {};
\node (C) at (1.4,1) {};
\node (D) at (1,2.05) {};
\node (E) at (0.1,2) {};
\node (F) at (-0.5,1) {};
\draw (A)-- (B) --(C) -- (D) --(E) --(F) --(A);
\node [ngreen, label={below left: $g$}] (g) at (0.2,1) {};
\draw (A) -- node[ngreen](AB){} (B)
      (C) -- node[ngreen] (CD){} (D);
\draw (AB) -- node[nred, label={below right: $r_1$}](r1){} (CD)
      (CD) -- node[nempty, pos=0.35, label={above: $p$}](p){} (g)
      (g) -- node[nblue, label={below left: $v_2$}](v2){} (AB);
\draw (r1) -- (g) node[nred,pos=1.6,label={above: $r_2$}](r2){};
\draw (g) -- (r2);
\node [nred, label={above: $r$}] (r) at (intersection of r1--p and D--E) {};
\draw (r1)--(p)--(r);
\draw (r) -- node[nempty, label={above left: $q$}](q){} (r2);
\end{tikzpicture}
\begin{tikzpicture}[line cap=round,line join=round,x=1.5cm,y=1.5cm,
label distance = 3 pt]
\clip(-1,-0.5) rectangle (3,2.4);
\tikzstyle{every node}=[inner sep=0pt]
\node (A) at (0,0.1) {};
\node (B) at (1.3,0) {};
\node (C) at (1.4,1) {};
\node (D) at (1,2.05) {};
\node (E) at (0.1,2) {};
\node (F) at (-0.5,1) {};
\draw (A)-- (B) --(C) -- (D) --(E) --(F) --(A);
\node [ngreen, label={below left: $g$}] (g) at (0.2,1) {};
\draw (A) -- node[ngreen](AB){} (B)
      (D) -- node[ngreen, pos=0.3] (DE){} (E);
\draw (AB) -- node[nred, pos=0.4, label={below right: $r_1$}](r1){} (DE)
      (DE) -- node[nempty, label={above left: $p$}](p){} (g)
      (g) -- node[nblue, label={below left: $v_2$}](v2){} (AB);
\draw (r1) -- (g) node[nred,pos=2.6,label={below: $r_2$}](r2){};
\draw (g) -- (r2);
\node [nred, label={above right: $r$}]
  (r) at (intersection of r2--p and C--D) {};
\draw (r2)--(p)--(r);
\draw (r) -- node[nempty, label={right: $q$}](q){} (r1);
\end{tikzpicture}
\end{center}
\end{proof}

\begin{defn}
Let $h(s)$ be the smallest number such that a planar point set of $s$ points in
general position contains an empty $s$-gon.
\end{defn}

It is known that $h(4)=5$, $h(5)=10$ \cite{Harborth1978}, and the best known
upper bound for $h(6)$ is 463 \cite{Koshelev2007}. Horton \cite{Horton1983}
showed that $h(s)=\infty$ for all $s\ge7$.

\begin{thm}
$mc_3(5) \leq 5h(6)-5 \leq 2310.$
\end{thm}

\begin{proof}
The proof will be by contradiction. Take a properly 5-colored point set $P$
that has at least $5h(6)-4$ points. It follows that the largest color class $C$
has at least $h(6)$ points. Since $\ell=3$, the points of $C$ are in general
position, so they contain an empty convex hexagon $H$. It follows that $H$ is
4-color blocked, which contradicts Theorem
\ref{main}.
\end{proof}

\section{Conclusions and remarks}

We have shown that empty convex hexagons cannot be 4-color-blocked, and with
this result we were able to derive the first upper bound for the value
$mc_3(5)$. We believe that similar techniques could be used to investigate
whether points in non-convex positions can be blocked by only a few colors, and
such an investigation could lead to resolving Conjecture \ref{blbchi}, or at
least a lot of progress in bounding the values of $mc_{\ell}(k)$.

However, these proofs should be automated. We believe that it is possible to
develop an algorithm that systematically checks all cases, using a search tree
that is kept relatively small with proper pruning techniques.

Another interesting question would be the relationship of the maximum clique
size and the chromatic number in visibility graphs. Is there a sequence $G_n$
of visibility graphs such that $\omega(G_n) = o(\chi(G_n))$? An answer to this
question would illustrate the relationship between Conjecture \ref{blbc} and
Conjecture \ref{blbchi}.

\bibliographystyle{plain}
\bibliography{combinatorialgeometry}

\pagebreak

\section*{Appendix}

\textit{Proof of Lemma \ref{convex4block}}\\
Let $Q = \{q_1,q_2,q_3,q_4\}$ denote the unicolored vertex set of the convex
quadrilateral.  Consider any blocking set $B$ of three colors, suppose that the
color of $Q$ is black and the colors of $B$ are red, green and blue. Let $B'$
denote $\conv(Q) \cap B$. 

On one hand, $B'$ is a properly colored set with 3 colors and no more than 3
points on a line, so by claim \ref{mc3-3}, $|B'| \leq 6$.  On the other hand,
$B'$ must have a point on any side of the quadrilateral (we denote them by
$y_{12},\ y_{23},\ y_{34},\ y_{41}$), and at least one point in $\iconv(Q)$ to
block the visibilities along the diagonals of the quadrilateral. So there are
two cases:

\noindent \textbf{Case 1: } $|B'|=5$. In this case, the only point in
$\iconv(Q)$ must block both diagonals, so it must be at the intersection of the
diagonals, denote it by $z$. $z$ is visible by all other points of $Q \cup B'$,
so it must have a unique color (green). Then all the $y_i$-s will be red or
blue. Visibility between $y_i$-s of neighbouring sides of the quadrilateral
cannot be blocked by any point, hence neighbouring $y_i$-s must have different
colors. So opposite pairs of $y_i$-s must have the same color, so they have to
be blocked by $z$. We got the figure above.

\noindent \textbf{Case 2: } $|B'|=6$. One of the two points in the interior is
visible by all the other points of $Q \cup B'$, it must have a unique color,
say green. Then the remaining five points must share two colors: red and blue.
Hence there is a unicolored triangle among them. If $B' \cup X$ is properly
colored, this triangle is blocked by only two colors, that contradicts lemma
\ref{triangle-blocking}.

What is left is to prove that there is no point of $B$ outside $\conv(X)$.
Suppose the contrary.  We may assume (w.l.o.g.) that $B \setminus B'$ has some
points in the convex territory bordered by rays $zq_1$ and $zq_2$. Let $p$ be a
point among those with a minimal distance to the line $q_1q_2$.  Now consider
the follwing pairs: $(p,x_1),\ (p,y_12),\ (p,z),\ (p,y_{23}),\ (p,y_{41})$.
Using the minimal distance property of $p$ we have that the first three pairs
are all visible ($y_12$ cannot block $(p,z)$ as that would mean $p,\ y_12,\ z$
and $y_34$ are all on the same line). One of the last two pairs may be blocked
by $y_{12}$ but then the other one will be a visible pair.  $x_1$ is black,
$y_{12}$ is red, $z$ is green and $y_{23},\ y_{41}$ are blue, so $p$ cannot
have any color, contradiction.  \qed

\bigskip
\noindent\textit{Proof of Lemma \ref{concave4block}}\\
Assume that the concave set $X=(x_1,x_2,x_3,x_4)$ is red. (Here
$x_4 \in iconv(x_1,x_2,x_3)$.)

Suppose there are 4 unicolored (blue) points in the blocking set in $\conv(X)
\setminus X$. It follows that there are 4 blue points $B=\{b_1,b_2,b_3,b_4\}$
that form a blue-empty set. Suppose the points of $B$ are in convex position.
By Lemma \ref{convex4block}, $B$ can only be blocked one way, but that
configuration does not contain a concave 4-point set, and it is also maximal.
Thus the points of $B$ are in concave position. But then there would be a
triangle $b_1,b_2,b_3 \in B$ such that $(\conv(b_1,b_2,b_3) \setminus
\{b_1,b_2,b_3\})\cap X = \emptyset$. (In particular only $x_4$ can be in
$\conv(B)$, but it can not be both in
$\iconv(x_1,x_2,x_4),\iconv(x_2,x_3,x_4),\iconv(x_3,x_1,x_4)$.)

Thus every color set in $\conv(X) \setminus X$ has at most 3 points. Since the
segments $(x_i,x_j)$ need to be blocked, there are at most 3 points in
$\iconv(T_1)\cup \iconv(T_2) \cup \iconv(T_3)$ (Here $T_i$ ($i=1,2,3$) denote
the red-empty triangles of $X$.) We distinguish 7 cases based on the
distribution of points in $T_1,T_2,T_3$. The shorthand notation $(a_1,a_2,a_3)$
means that $|T_i\cap S|=a_i$. From the discussion above it follows that
$a_1+a_2+a_3 \leq 3$.  We may assume without loss of generality that $a_1 \leq
a_2 \leq a_3$.

We denote by $s_{ij}$ the point lying on the segment $x_ix_j$, and let
$S_k=\{s_{ij} | 1 \leq i \leq j \leq 4, i,j,k \text{ are distinct}\}$. Note
that beams exists if and only if $a_i=1$ or $a_i=2$ by Lemma
\ref{triangle-blocking}.

\noindent \textbf{Case 1: } $(0,0,0)$\\
The three points in $s_{i4} \;(i=1,2,3)$ can all see each other, thus their
colors are distinct. Since each $T_i$ contains at least on point of each color,
there must be 2 points in each color class.  it follows that the blockers of
the opposing edges of the tetrahedron have the same color. Since $\ell=3$,
$s_{i4}$ and $s_{jk}$ can not be blocked by $x_4$, so they are blocked by
either $s_{j4}$ or $s_{k4}$.  (Here $\{i,j,k\}=\{1,2,3\}$.) It is easy to
observe that we arrive at the configuration defined in the statement of the
lemma (or its reflection).

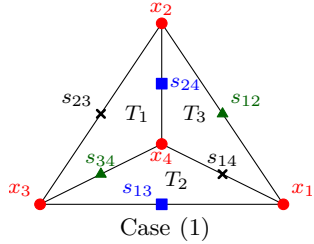
\begin{center}
\begin{tikzpicture}[line cap=round,line join=round,x=0.8cm,y=0.8cm]
\clip(-2.5,-1.13) rectangle (2.46,2.4);
\draw (2,-1)-- (0,2);
\draw (-2,-1)-- (2,-1);
\draw (0,2)-- (-2,-1);
\draw (0,0)-- (0,2);
\draw (0,0)-- (2,-1);
\draw (0,0)-- (-2,-1);
\begin{scriptsize}
\fill [sred] plot coordinates{(2,-1)} node[above right] {$x_1$};
\fill [sred] plot coordinates{(0,2)} node[above] {$x_2$};
\fill [sred] plot coordinates{(-2,-1)} node[above left] {$x_3$};
\fill [sred] plot coordinates{(0,0)} node[below] {$x_4$};
\fill [sblue] plot coordinates{(0,1)} node[right] {$s_{24}$};
\draw [sblack] plot coordinates{(1,-0.5)} node[above] {$s_{14}$};
\fill [sgreen] plot coordinates{(1,0.5)} node[above right] {$s_{12}$};

\fill [sblue] plot coordinates{(0,-1)} node[above left] {$s_{13}$};
\fill [sgreen] plot coordinates{(-1,-0.5)} node[above] {$s_{34}$};
\draw [sblack] plot coordinates{(-1,0.5)} node[above left] {$s_{23}$};

\node [black] at (0.55,0.5) {$T_3$};
\node [black] at (-0.4,0.5) {$T_1$};
\node [black] at (0.25,-0.6) {$T_2$};
\end{scriptsize}
\end{tikzpicture}
\vspace{-0.4cm}
\captionof{figure}{Case (1)}
\end{center}

\noindent \textbf{Case 2: } $(0,0,1)$\\
Let $p$ be the blocking point of the beam of $T_3$. The blocking point can not
lay on the line $x_3x_4$ since there would be 4 collinear points on the line.
Thus it lies on one side of line $x_3x_4$. We distinguish two subcases based on
the position of the beam in $T_3$.

(2a) If the beam is on $s_{14}$ and $s_{24}$, then we may assume that $p$ lies
on the left side of the ray $\overrightarrow{x_3x_4}$. Note that $s_{24}$ can
not block any points from $p$ (again, because $\ell=3$). Thus $p$ can see both
$s_{24},s_{23}$ and $s_{34}$, three points with different colors, which is a
contradiction.

(2b) If the beam is on $s_{14}$ and $s_{12}$, and $p$ lies on the right side of
$\overrightarrow{x_3x_4}$, then a similar argument shows that $p$ can see all
points in $S_2$.  Otherwise since $s_{24}$ and $s_{34}$ can not be blue (since
$T_3$ and $T_2$ already has its blue points), $s_{23}$ is blue. It follows that
$s_{34}$ has the same color as $p$.  Now the segments $s_{23}s_{12}$ and
$ps_{34}$ can only be blocked by $s_{24}$, but since the segments are disjoint,
it can not block both of them.

Note that the third possible beam position can be handled like this because it
can be obtained by a reflection from this case.

\begin{center}
\noindent\begin{minipage}{0.4\textwidth}
\begin{center}
\begin{tikzpicture}[line cap=round,line join=round,x=0.8cm,y=0.8cm]
\clip(-2.5,-1.13) rectangle (2.46,2.4);
\draw (2,-1)-- (0,2);
\draw (-2,-1)-- (2,-1);
\draw (0,2)-- (-2,-1);
\draw (0,0)-- (0,2);
\draw (0,0)-- (2,-1);
\draw (0,0)-- (-2,-1);
\draw (1,-0.5)-- (0,1);
\begin{scriptsize}
\fill [sred] plot coordinates{(2,-1)} node[above right] {$x_1$};
\fill [sred] plot coordinates{(0,2)} node[above] {$x_2$};
\fill [sred] plot coordinates{(-2,-1)} node[above left] {$x_3$};
\fill [sred] plot coordinates{(0,0)} node[below] {$x_4$};
\fill [sblue] plot coordinates{(0,1)} node[left] {$s_{24}$};
\fill [sblue] plot coordinates{(1,-0.5)} node[below left] {$s_{14}$};
\draw [sblack] plot coordinates{(0.4,0.4)} node[above right] {$p$};
\end{scriptsize}
\end{tikzpicture}
\end{center}
\vspace{-0.5cm} \captionof{figure}{Case (2a)}
\end{minipage}
\hspace{0.5cm}
\noindent\begin{minipage}{0.4\textwidth}
\begin{center}
\begin{tikzpicture}[line cap=round,line join=round,x=0.8cm,y=0.8cm]
\clip(-2.5,-1.13) rectangle (2.46,2.4);
\draw (2,-1)-- (0,2);
\draw (-2,-1)-- (2,-1);
\draw (0,2)-- (-2,-1);
\draw (0,0)-- (0,2);
\draw (0,0)-- (2,-1);
\draw (0,0)-- (-2,-1);
\draw (1,-0.5)-- (0.5,1.25);
\begin{scriptsize}
\fill [sred] plot coordinates{(2,-1)} node[above right] {$x_1$};
\fill [sred] plot coordinates{(0,2)} node[above] {$x_2$};
\fill [sred] plot coordinates{(-2,-1)} node[above left] {$x_3$};
\fill [sred] plot coordinates{(0,0)} node[below] {$x_4$};
\fill [sblue] plot coordinates{(0.5,1.25)} node[above right] {$s_{12}$};
\fill [sblue] plot coordinates{(1,-0.5)} node[below left] {$s_{14}$};
\fill [sblue] plot coordinates{(-1,0.5)} node[above left] {$s_{23}$};
\draw [sblack] plot coordinates{(-1,-0.5)} node[below right] {$s_{34}$};
\draw [sblack] plot coordinates{(0.65,0.7)} node[left] {$p$};
\end{scriptsize}
\end{tikzpicture}
\end{center}
\vspace{-0.5cm} \captionof{figure}{Case (2b)}
\end{minipage}
\end{center}

\noindent \textbf{Case 3: } $(0,0,2)$\\
Again, we have 4 subcases based on the position of the beam in $T_3$, and the
way $T_3$ is blocked. Let $p$ be the midpoint of the beam, and let $q$ be the
other point in $T_3$. Note that $s_{14}$, $s_{24}$ and $q$ can not block
anything from $p$ because $\ell=3$. If $p$ is on the right side of the ray
$\overrightarrow{x_3x_4}$, then it can see all points in $S_2$, one of which
has the same color as $p$. If $p$ is on the left side, it will see all points
in $S_1$.

\begin{center}
\noindent\begin{minipage}{0.4\textwidth}
\begin{center}
\begin{tikzpicture}[line cap=round,line join=round,x=0.9cm,y=0.9cm]
\clip(-2.5,-1.13) rectangle (2.46,2.4);
\draw (2,-1)-- (0,2);
\draw (-2,-1)-- (2,-1);
\draw (0,2)-- (-2,-1);
\draw (0,0)-- (0,2);
\draw (0,0)-- (2,-1);
\draw (0,0)-- (-2,-1);
\draw (1,-0.5)-- (0,1);
\draw (0.4,0.4) -- (1.05,0.4);
\begin{scriptsize}
\fill [sred] plot coordinates{(2,-1)} node[above right] {$x_1$};
\fill [sred] plot coordinates{(0,2)} node[above] {$x_2$};
\fill [sred] plot coordinates{(-2,-1)} node[above left] {$x_3$};
\fill [sred] plot coordinates{(0,0)} node[below] {$x_4$};
\fill [sblue] plot coordinates{(0,1)} node[left] {$s_{24}$};
\fill [sblue] plot coordinates{(1,-0.5)} node[below left] {$s_{14}$};
\draw [sblack] plot coordinates{(0.4,0.4)} node[below] {$p$};
\draw [sblack] plot coordinates{(1.05,0.4)} node[above right] {$s_{12}$};
\draw [sgreen] plot coordinates{(0.725,0.4)} node[above] {$q$};
\end{scriptsize}
\end{tikzpicture}
\end{center}
\vspace{-0.5cm} \captionof{figure}{Case (3a)}
\end{minipage}
\hspace{0.5cm}
\noindent\begin{minipage}{0.4\textwidth}
\begin{center}
\begin{tikzpicture}[line cap=round,line join=round,x=0.9cm,y=0.9cm]
\clip(-2.5,-1.13) rectangle (2.46,2.4);
\draw (2,-1)-- (0,2);
\draw (-2,-1)-- (2,-1);
\draw (0,2)-- (-2,-1);
\draw (0,0)-- (0,2);
\draw (0,0)-- (2,-1);
\draw (0,0)-- (-2,-1);
\draw (1,-0.5) -- (1,0.5);
\draw (1,0) -- (0,0.5);
\begin{scriptsize}
\fill [sred] plot coordinates{(2,-1)} node[above right] {$x_1$};
\fill [sred] plot coordinates{(0,2)} node[above] {$x_2$};
\fill [sred] plot coordinates{(-2,-1)} node[above left] {$x_3$};
\fill [sred] plot coordinates{(0,0)} node[below] {$x_4$};
\fill [sblue] plot coordinates{(1,0.5)} node[above right] {$s_{12}$};
\fill [sblue] plot coordinates{(1,-0.5)} node[below left] {$s_{14}$};
\draw [sblack] plot coordinates{(0,0.5)} node[left] {$s_{24}$};
\draw [sblack] plot coordinates{(1,0)} node[below left] {$p$};
\draw [sgreen] plot coordinates{(0.5,0.25)} node[above] {$q$};
\end{scriptsize}
\end{tikzpicture}
\end{center}
\vspace{-0.5cm} \captionof{figure}{Case (3b)}
\end{minipage}

\noindent\begin{minipage}{0.4\textwidth}
\begin{center}
\begin{tikzpicture}[line cap=round,line join=round,x=0.9cm,y=0.9cm]
\clip(-2.5,-1.13) rectangle (2.46,2.4);
\draw (2,-1)-- (0,2);
\draw (-2,-1)-- (2,-1);
\draw (0,2)-- (-2,-1);
\draw (0,0)-- (0,2);
\draw (0,0)-- (2,-1);
\draw (0,0)-- (-2,-1);
\draw (1.5,-0.75)-- (0,1.5);
\draw (0.4,0.4) -- (1.05,0.4);
\begin{scriptsize}
\fill [sred] plot coordinates{(2,-1)} node[above right] {$x_1$};
\fill [sred] plot coordinates{(0,2)} node[above] {$x_2$};
\fill [sred] plot coordinates{(-2,-1)} node[above left] {$x_3$};
\fill [sred] plot coordinates{(0,0)} node[below] {$x_4$};
\fill [sblue] plot coordinates{(0,1.5)} node[below left] {$s_{24}$};
\fill [sblue] plot coordinates{(1.5,-0.75)} node[left] {$s_{14}$};
\draw [sblack] plot coordinates{(0.4,0.4)} node[below] {$q$};
\draw [sblack] plot coordinates{(1.05,0.4)} node[above right] {$s_{12}$};
\draw [sgreen] plot coordinates{(0.725,0.4)} node[above] {$p$};
\end{scriptsize}
\end{tikzpicture}
\end{center}
\vspace{-0.5cm} \captionof{figure}{Case (3c)}
\end{minipage}
\hspace{0.5cm}
\noindent\begin{minipage}{0.4\textwidth}
\begin{center}
\begin{tikzpicture}[line cap=round,line join=round,x=0.9cm,y=0.9cm]
\clip(-2.5,-1.13) rectangle (2.46,2.4);
\draw (2,-1)-- (0,2);
\draw (-2,-1)-- (2,-1);
\draw (0,2)-- (-2,-1);
\draw (0,0)-- (0,2);
\draw (0,0)-- (2,-1);
\draw (0,0)-- (-2,-1);
\draw (0.5,-0.25) -- (0.5,1.25);
\draw (1,0) -- (0,0.5);
\begin{scriptsize}
\fill [sred] plot coordinates{(2,-1)} node[above right] {$x_1$};
\fill [sred] plot coordinates{(0,2)} node[above] {$x_2$};
\fill [sred] plot coordinates{(-2,-1)} node[above left] {$x_3$};
\fill [sred] plot coordinates{(0,0)} node[below] {$x_4$};
\fill [sblue] plot coordinates{(0.5,1.25)} node[above right] {$s_{12}$};
\fill [sblue] plot coordinates{(0.5,-0.25)} node[below left] {$s_{14}$};
\draw [sblack] plot coordinates{(0,0.5)} node[left] {$s_{24}$};
\draw [sblack] plot coordinates{(1,0)} node[below left] {$q$};
\draw [sgreen] plot coordinates{(0.5,0.25)} node[above right] {$p$};
\end{scriptsize}
\end{tikzpicture}
\end{center}
\vspace{-0.5cm} \captionof{figure}{Case (3d)}
\end{minipage}
\end{center}

\noindent \textbf{Case 4: } $(0,0,3)$\\
If $p$ is on the right side of $\overrightarrow{x_3x_4}$, then it can see all
points in $S_2$.  Otherwise $p$ and $q$ are both on the left side, but then $q$
can see all points in $S_1$.
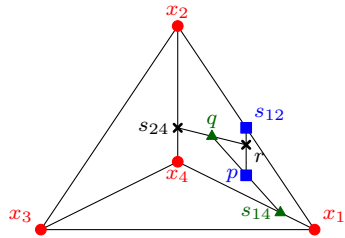
\begin{center}
\begin{tikzpicture}[line cap=round,line join=round,x=0.9cm,y=0.9cm]
\clip(-2.5,-1.13) rectangle (2.46,2.4);
\draw (2,-1)-- (0,2);
\draw (-2,-1)-- (2,-1);
\draw (0,2)-- (-2,-1);
\draw (0,0)-- (0,2);
\draw (0,0)-- (2,-1);
\draw (0,0)-- (-2,-1);
\draw (1,-0.2) -- (1,0.5);
\draw (1,0.25) -- (0,0.5);
\draw (0.5,0.375) -- (1.5,-0.75);
\begin{scriptsize}
\fill [sred] plot coordinates{(2,-1)} node[above right] {$x_1$};
\fill [sred] plot coordinates{(0,2)} node[above] {$x_2$};
\fill [sred] plot coordinates{(-2,-1)} node[above left] {$x_3$};
\fill [sred] plot coordinates{(0,0)} node[below] {$x_4$};
\fill [sblue] plot coordinates{(1,0.5)} node[above right] {$s_{12}$};
\fill [sgreen] plot coordinates{(1.5,-0.75)} node[left] {$s_{14}$};
\draw [sblack] plot coordinates{(0,0.5)} node[left] {$s_{24}$};
\draw [sblue] plot coordinates{(1,-0.2)} node[left] {$p$};
\draw [sgreen] plot coordinates{(0.5,0.375)} node[above] {$q$};
\draw [sblack] plot coordinates{(1,0.25)} node[below right] {$r$};
\end{scriptsize}
\end{tikzpicture}
\vspace{-0.2cm} \captionof{figure}{Case (4)}
\end{center}
\noindent \textbf{Case 5: } $(0,1,1)$\\
We have 6 subcases based on beam positions.

(5a) Beams are $s_{13}s_{14}$ and $s_{14}s_{12}$. It follows that $s_{13},
\;s_{14}$ and $s_{12}$ have the same color (blue). But triangle $T_1$ also
needs a blue point, so we have 4 blue points, contradiction.

(5b) Beams are $s_{24}s_{14}$ and $s_{14}s_{34}$. But then $s_{34}$ and
$s_{24}$ have the same color and they can see each other since $int(T_1)$ has
no points.

(5c) Beams are $s_{13}s_{34}$ and $s_{12}s_{24}$. The endpoints of the two
beams must have distinct colors because $s_{34}$ and $s_{24}$ must be different
to block $T_1$.  Assume the endpoints are blue and black. It is easy to observe
that $s_{14}$ has to be green, thus the point $p$ in $int(T_3)$ has to be blue.
If $p$ lies on the left side of $\overrightarrow{x_3x_4}$, then it can see
$s_{34}$, otherwise it has to be blocked from $s_{34}$ and $s_{13}$, but the
only point we can use for blocking is $s_14$, so $p$ will see at least one of
$s_{34}$ and $s_{13}$.

(5d) Beams are $s_{14}s_{34}$ and $s_{12}s_{14}$. Let $q$ be the blocker of
beam $s_{14}s_{34}$. Either $s_{24}$ is black, in which case it can see $q$, or
there are $2$ black points in $T_1 \cup \conv(x_2x_4s_{14}s_{12})$, one of which
can see $q$, because $q$ can be blocked from them only by $x_4$.

(5e) Beams are $s_{14}s_{24}$ and $s_{13}s_{34}$. Let $p$ be the blocker of
beam $s_{14}s_{24}$. Now depending on which side of $\overrightarrow{x_3x_4}$
point $p$ is, $p$ can see all points in $T_1$ or all points in
$\conv(x_4s_{34}s_{13}x_1)$.

(5f) Beams are $s_{34}s_{13}$ and $s_{14}s_{12}$. If $p$ is on the right side
of $\overrightarrow{x_3x_4}$, then it can see all points in
$\conv(x_4s_{34}s_{13}x_1)$.  Otherwise $s_24$ can only block one segment out of
$s_{34}p$ and $s_{23}s_{12}$.

\noindent Note that the cases (5d), (5e), (5f) have reflections which can be
handled the same way.

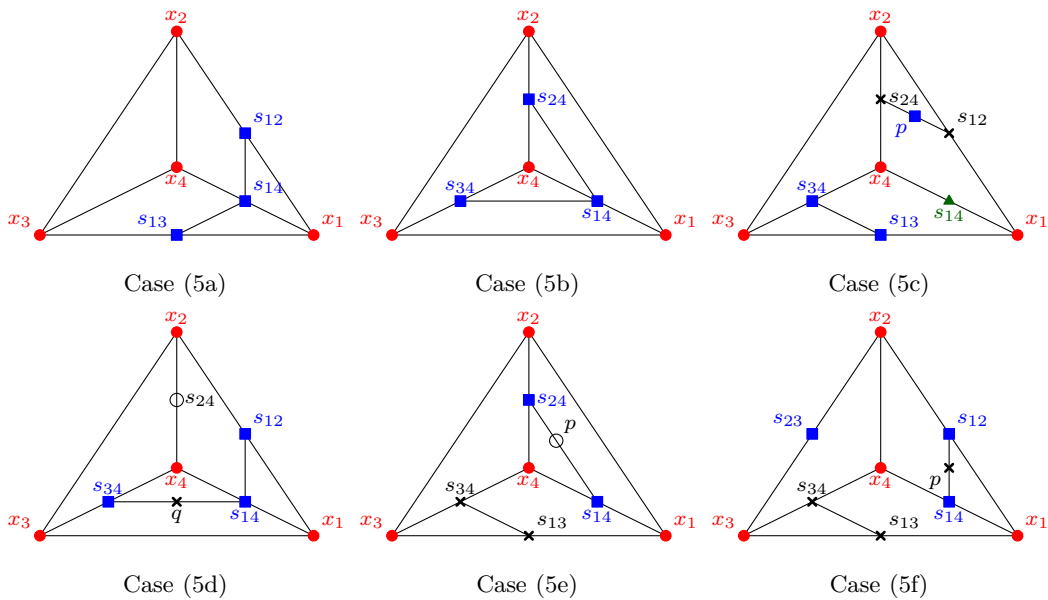
\begin{center}
\noindent\begin{minipage}{0.3\textwidth}
\begin{center}
\begin{tikzpicture}[line cap=round,line join=round,x=0.9cm,y=0.9cm]
\clip(-2.5,-1.13) rectangle (2.46,2.4);
\draw (2,-1)-- (0,2);
\draw (-2,-1)-- (2,-1);
\draw (0,2)-- (-2,-1);
\draw (0,0)-- (0,2);
\draw (0,0)-- (2,-1);
\draw (0,0)-- (-2,-1);
\draw (0,-1) -- (1,-0.5) -- (1,0.5);
\begin{scriptsize}
\fill [sred] plot coordinates{(2,-1)} node[above right] {$x_1$};
\fill [sred] plot coordinates{(0,2)} node[above] {$x_2$};
\fill [sred] plot coordinates{(-2,-1)} node[above left] {$x_3$};
\fill [sred] plot coordinates{(0,0)} node[below] {$x_4$};
\draw [sblue] plot coordinates{(1,-0.5)} node[above right] {$s_{14}$};
\fill [sblue] plot coordinates{(1,0.5)} node[above right] {$s_{12}$};
\fill [sblue] plot coordinates{(0,-1)} node[above left] {$s_{13}$};
\end{scriptsize}
\end{tikzpicture}
\end{center}
\vspace{-0.5cm} \captionof{figure}{Case (5a)}
\end{minipage}
\noindent\begin{minipage}{0.3\textwidth}
\begin{center}
\begin{tikzpicture}[line cap=round,line join=round,x=0.9cm,y=0.9cm]
\clip(-2.5,-1.13) rectangle (2.46,2.4);
\draw (2,-1)-- (0,2);
\draw (-2,-1)-- (2,-1);
\draw (0,2)-- (-2,-1);
\draw (0,0)-- (0,2);
\draw (0,0)-- (2,-1);
\draw (0,0)-- (-2,-1);
\draw (-1,-0.5) -- (1,-0.5) -- (0,1);
\begin{scriptsize}
\fill [sred] plot coordinates{(2,-1)} node[above right] {$x_1$};
\fill [sred] plot coordinates{(0,2)} node[above] {$x_2$};
\fill [sred] plot coordinates{(-2,-1)} node[above left] {$x_3$};
\fill [sred] plot coordinates{(0,0)} node[below] {$x_4$};
\fill [sblue] plot coordinates{(-1,-0.5)} node[above] {$s_{34}$};
\fill [sblue] plot coordinates{(0,1)} node[right] {$s_{24}$};
\draw [sblue] plot coordinates{(1,-0.5)} node[below] {$s_{14}$};
\end{scriptsize}
\end{tikzpicture}
\end{center}
\vspace{-0.5cm} \captionof{figure}{Case (5b)}
\end{minipage}
\noindent\begin{minipage}{0.3\textwidth}
\begin{center}
\begin{tikzpicture}[line cap=round,line join=round,x=0.9cm,y=0.9cm]
\clip(-2.5,-1.13) rectangle (2.46,2.4);
\draw (2,-1)-- (0,2);
\draw (-2,-1)-- (2,-1);
\draw (0,2)-- (-2,-1);
\draw (0,0)-- (0,2);
\draw (0,0)-- (2,-1);
\draw (0,0)-- (-2,-1);
\draw (-1,-0.5) -- (0,-1);
\draw (0,1) -- (1,0.5);
\begin{scriptsize}
\fill [sred] plot coordinates{(2,-1)} node[above right] {$x_1$};
\fill [sred] plot coordinates{(0,2)} node[above] {$x_2$};
\fill [sred] plot coordinates{(-2,-1)} node[above left] {$x_3$};
\fill [sred] plot coordinates{(0,0)} node[below] {$x_4$};
\fill [sblue] plot coordinates{(-1,-0.5)} node[above] {$s_{34}$};
\fill [sblack] plot coordinates{(0,1)} node[right] {$s_{24}$};
\fill [sblack] plot coordinates{(1,0.5)} node[above right] {$s_{12}$};
\fill [sblue] plot coordinates{(0,-1)} node[above right] {$s_{13}$};
\fill [sgreen] plot coordinates{(1,-0.5)} node[below] {$s_{14}$};
\fill [sblue] plot coordinates{(0.5,0.75)} node[below left] {$p$};
\end{scriptsize}
\end{tikzpicture}
\end{center}
\vspace{-0.5cm} \captionof{figure}{Case (5c)}
\end{minipage}

\noindent\begin{minipage}{0.3\textwidth}
\begin{center}
\begin{tikzpicture}[line cap=round,line join=round,x=0.9cm,y=0.9cm]
\clip(-2.5,-1.13) rectangle (2.46,2.4);
\draw (2,-1)-- (0,2);
\draw (-2,-1)-- (2,-1);
\draw (0,2)-- (-2,-1);
\draw (0,0)-- (0,2);
\draw (0,0)-- (2,-1);
\draw (0,0)-- (-2,-1);
\draw (-1,-0.5) -- (1,-0.5) -- (1,0.5);
\begin{scriptsize}
\fill [sred] plot coordinates{(2,-1)} node[above right] {$x_1$};
\fill [sred] plot coordinates{(0,2)} node[above] {$x_2$};
\fill [sred] plot coordinates{(-2,-1)} node[above left] {$x_3$};
\fill [sred] plot coordinates{(0,0)} node[below] {$x_4$};
\draw [sblue] plot coordinates{(1,-0.5)} node[below] {$s_{14}$};
\fill [sblue] plot coordinates{(1,0.5)} node[above right] {$s_{12}$};
\fill [sblue] plot coordinates{(-1,-0.5)} node[above] {$s_{34}$};
\fill [sblack] plot coordinates{(0,-0.5)} node[below] {$q$};
\draw [smarko] plot coordinates{(0,1)} node[right] {$s_{24}$};
\end{scriptsize}
\end{tikzpicture}
\end{center}
\vspace{-0.5cm} \captionof{figure}{Case (5d)}
\end{minipage}
\noindent\begin{minipage}{0.3\textwidth}
\begin{center}
\begin{tikzpicture}[line cap=round,line join=round,x=0.9cm,y=0.9cm]
\clip(-2.5,-1.13) rectangle (2.46,2.4);
\draw (2,-1)-- (0,2);
\draw (-2,-1)-- (2,-1);
\draw (0,2)-- (-2,-1);
\draw (0,0)-- (0,2);
\draw (0,0)-- (2,-1);
\draw (0,0)-- (-2,-1);
\draw (-1,-0.5) -- (0,-1);
\draw (1,-0.5) -- (0,1);
\begin{scriptsize}
\fill [sred] plot coordinates{(2,-1)} node[above right] {$x_1$};
\fill [sred] plot coordinates{(0,2)} node[above] {$x_2$};
\fill [sred] plot coordinates{(-2,-1)} node[above left] {$x_3$};
\fill [sred] plot coordinates{(0,0)} node[below] {$x_4$};
\fill [sblue] plot coordinates{(0,1)} node[right] {$s_{24}$};
\draw [sblue] plot coordinates{(1,-0.5)} node[below] {$s_{14}$};
\fill [sblack] plot coordinates{(-1,-0.5)} node[above] {$s_{34}$};
\fill [sblack] plot coordinates{(0,-1)} node[above right] {$s_{13}$};
\draw [smarko] plot coordinates{(0.4,0.4)} node[above right] {$p$};
\end{scriptsize}
\end{tikzpicture}
\end{center}
\vspace{-0.5cm} \captionof{figure}{Case (5e)}
\end{minipage}
\noindent\begin{minipage}{0.3\textwidth}
\begin{center}
\begin{tikzpicture}[line cap=round,line join=round,x=0.9cm,y=0.9cm]
\clip(-2.5,-1.13) rectangle (2.46,2.4);
\draw (2,-1)-- (0,2);
\draw (-2,-1)-- (2,-1);
\draw (0,2)-- (-2,-1);
\draw (0,0)-- (0,2);
\draw (0,0)-- (2,-1);
\draw (0,0)-- (-2,-1);
\draw (-1,-0.5) -- (0,-1);
\draw (1,-0.5) -- (1,0.5);
\begin{scriptsize}
\fill [sred] plot coordinates{(2,-1)} node[above right] {$x_1$};
\fill [sred] plot coordinates{(0,2)} node[above] {$x_2$};
\fill [sred] plot coordinates{(-2,-1)} node[above left] {$x_3$};
\fill [sred] plot coordinates{(0,0)} node[below] {$x_4$};
\fill [sblack] plot coordinates{(-1,-0.5)} node[above] {$s_{34}$};
\fill [sblue] plot coordinates{(1,0.5)} node[above right] {$s_{12}$};
\fill [sblack] plot coordinates{(0,-1)} node[above right] {$s_{13}$};
\fill [sblue] plot coordinates{(1,-0.5)} node[below] {$s_{14}$};
\fill [sblue] plot coordinates{(-1,0.5)} node[above left] {$s_{23}$};
\draw [sblack] plot coordinates{(1,0)} node[below left] {$p$};
\end{scriptsize}
\end{tikzpicture}
\end{center}
\vspace{-0.5cm} \captionof{figure}{Case (5f)}
\end{minipage}
\end{center}

\noindent \textbf{Case 6: } $(0,1,2)$\\
We have 9 subcases based on beam poistions. Let $p$ be the midpoint of the beam
in $T_3$.  Notice that in the previous case, all subcases except (5c) and (5d)
ended with the conclusion that $p$ can see all points in $T_1$ or all points in
$T_2$. All these arguments can be carried over here, since the new point $q$ in
this case cannot block any of these points from $p$. It is also easy to see
that the argument in (5d) and its reflection works here as well. So the only
remaining case is corresponding to the beam position (5c).

(6c) Beams are $s_{13}s_{34}$ and $s_{12}s_{24}$. The endpoints of the two
beams have different colors, suppose that $s_{13},\; s_{34}$ are blue and
$s_{12},\; s_{24}$ are black.  The colros of the rest of the points are
determined as shown in the figures. If $p$ lies on the left side of
$\overrightarrow{x_3x_4}$, then it can see all points in $S_1$ otherwise
$s_{13}$ and $s_{34}$ can see the blue point in $T_3$.

%
%
%
%

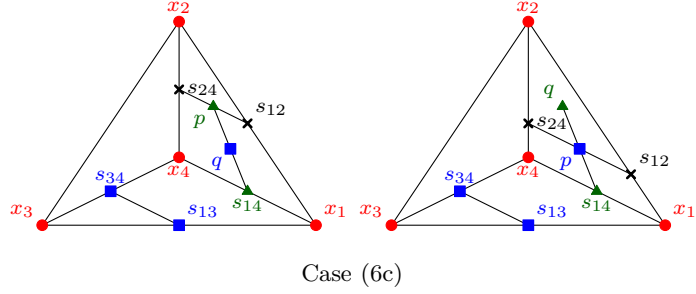
\begin{center}
\noindent\begin{minipage}{0.8\textwidth}
\begin{center}
\begin{tikzpicture}[line cap=round,line join=round,x=0.9cm,y=0.9cm]
\clip(-2.5,-1.13) rectangle (2.46,2.4);
\draw (2,-1)-- (0,2);
\draw (-2,-1)-- (2,-1);
\draw (0,2)-- (-2,-1);
\draw (0,0)-- (0,2);
\draw (0,0)-- (2,-1);
\draw (0,0)-- (-2,-1);
\draw (-1,-0.5) -- (0,-1);
\draw (0,1) -- (1,0.5);
\draw (1,-0.5) -- (0.5,0.75);
\begin{scriptsize}
\fill [sred] plot coordinates{(2,-1)} node[above right] {$x_1$};
\fill [sred] plot coordinates{(0,2)} node[above] {$x_2$};
\fill [sred] plot coordinates{(-2,-1)} node[above left] {$x_3$};
\fill [sred] plot coordinates{(0,0)} node[below] {$x_4$};
\fill [sblue] plot coordinates{(-1,-0.5)} node[above] {$s_{34}$};
\fill [sblack] plot coordinates{(0,1)} node[right] {$s_{24}$};
\fill [sblack] plot coordinates{(1,0.5)} node[above right] {$s_{12}$};
\fill [sblue] plot coordinates{(0,-1)} node[above right] {$s_{13}$};
\fill [sgreen] plot coordinates{(1,-0.5)} node[below](s14) {$s_{14}$};
\fill [sgreen] plot coordinates{(0.5,0.75)} node[below left] {$p$};
\fill [sblue] plot coordinates{(0.75,0.125)} node[below left] {$q$};
\end{scriptsize}
\end{tikzpicture} 
\begin{tikzpicture}[line cap=round,line join=round,x=0.9cm,y=0.9cm]
\clip(-2.5,-1.13) rectangle (2.46,2.4);
\draw (2,-1)-- (0,2);
\draw (-2,-1)-- (2,-1);
\draw (0,2)-- (-2,-1);
\draw (0,0)-- (0,2);
\draw (0,0)-- (2,-1);
\draw (0,0)-- (-2,-1);
\draw (-1,-0.5) -- (0,-1);
\draw (0,0.5) -- (1.5,-0.25);
\draw (1,-0.5) -- (0.5,0.75);
\begin{scriptsize}
\fill [sred] plot coordinates{(2,-1)} node[above right] {$x_1$};
\fill [sred] plot coordinates{(0,2)} node[above] {$x_2$};
\fill [sred] plot coordinates{(-2,-1)} node[above left] {$x_3$};
\fill [sred] plot coordinates{(0,0)} node[below] {$x_4$};
\fill [sblue] plot coordinates{(-1,-0.5)} node[above] {$s_{34}$};
\fill [sblack] plot coordinates{(0,0.5)} node[right] {$s_{24}$};
\fill [sblack] plot coordinates{(1.5,-0.25)} node[above right] {$s_{12}$};
\fill [sblue] plot coordinates{(0,-1)} node[above right] {$s_{13}$};
\fill [sgreen] plot coordinates{(1,-0.5)} node[below](s14) {$s_{14}$};
\fill [sgreen] plot coordinates{(0.5,0.75)} node[above left] {$q$};
\fill [sblue] plot coordinates{(0.75,0.125)} node[below left] {$p$};
\end{scriptsize}
\end{tikzpicture}
\end{center}
\vspace{-0.5cm} \captionof{figure}{Case (6c)}
\end{minipage}
\end{center}

\noindent \textbf{Case 7: } $(1,1,1)$

(7a) The beams form a path of length 3. In this case the beam endpoints have
the same color, and there is 4 of them. (With rotations and reflections, this
case covers 9 beam positionings.)

(7b) The beams form a circle of length 3. The blue triangle
$s_{14}s_{24}s_{34}$ has one point inside, so it must be blocked as option 2 in
Lemma \ref{triangle-blocking}. Without loss of generality we may suppose that
$q$ and $r$ have the same color. They need to be blocked from $s_{12}$, but the
only point that can block either of them is $p$, and it cannot block both
$s_{12}q$ and $s_{12}r$.

(7c) The beams form a path of length two, and one of the endpoints of the path
is not on the sides of $x_1x_2x_3$. We use the notations of the figure. Point
$p$ hast to be black or green. Depending on its position in relation to ray
$\overrightarrow{x_1x_4}$ it can see all green and black points in $T_3$ or
$T_2$. (This case corresponds to 6 beam positions after taking rotations and
reflections.)

(7d) The beams form a path of length two, and both endpoints of the path are on
the sides of $x_1x_2x_3$. In this case the path is disjoint from triangle
$T_1$, but $T_1$ must contain at least one more blue point, so we would have 4
blue points. (This case corresponds to $9$ beam positions.)

(7e) The endpoints of the beams are disjoint. We use the notations of the
figure.  If $q$ is on the right side of $\overrightarrow{x_3x_4}$, then it can
see $s_34$, otherwise $s_{12}$ is on the left side, and it can see $r$.

\begin{center}
\noindent\begin{minipage}{0.45\textwidth}
\begin{center}
\begin{tikzpicture}[line cap=round,line join=round,x=0.9cm,y=0.9cm]
\clip(-2.5,-1.13) rectangle (2.46,2.4);
\draw (2,-1)-- (0,2);
\draw (-2,-1)-- (2,-1);
\draw (0,2)-- (-2,-1);
\draw (0,0)-- (0,2);
\draw (0,0)-- (2,-1);
\draw (0,0)-- (-2,-1);
\draw (-1,-0.5) -- (1,-0.5) -- (0,1) -- (-1,-0.5);
\draw (-0.37,0.45) -- (0.4,-0.5);
\begin{scriptsize}
\fill [sred] plot coordinates{(2,-1)} node[above right] {$x_1$};
\fill [sred] plot coordinates{(0,2)} node[above] {$x_2$};
\fill [sred] plot coordinates{(-2,-1)} node[above left] {$x_3$};
\fill [sred] plot coordinates{(0,0)} node[right] {$x_4$};
\fill [sblue] plot coordinates{(-1,-0.5)} node[below] {$s_{34}$};
\fill [sblue] plot coordinates{(0,1)} node[right] {$s_{24}$};
\fill [sblue] plot coordinates{(1,-0.5)} node[below](s14) {$s_{14}$};
\fill [sblack] plot coordinates{(1.05,0.4)} node[above right] {$s_{12}$};
\fill [sgreen] plot coordinates{(0.5,0.25)} node[above right] {$p$};
\fill [sblack] plot coordinates{(-0.37,0.45)} node[left] {$q$};
\fill [sblack] plot coordinates{(0.4,-0.5)} node[below left] {$r$};
\end{scriptsize}
\end{tikzpicture}
\end{center}
\vspace{-0.5cm} \captionof{figure}{Case (7b)}
\end{minipage}
\noindent\begin{minipage}{0.45\textwidth}
\begin{center}
\begin{tikzpicture}[line cap=round,line join=round,x=0.9cm,y=0.9cm]
\clip(-2.5,-1.13) rectangle (2.46,2.4);
\draw (2,-1)-- (0,2);
\draw (-2,-1)-- (2,-1);
\draw (0,2)-- (-2,-1);
\draw (0,0)-- (0,2);
\draw (0,0)-- (2,-1);
\draw (0,0)-- (-2,-1);
\draw (0,1) -- (-1,-0.5);
\draw (1,-0.5) -- (1,0.5);
\draw (-1,-0.5) -- (0,-1);
\begin{scriptsize}
\fill [sred] plot coordinates{(2,-1)} node[above right] {$x_1$};
\fill [sred] plot coordinates{(0,2)} node[above] {$x_2$};
\fill [sred] plot coordinates{(-2,-1)} node[above left] {$x_3$};
\fill [sred] plot coordinates{(0,0)} node[below] {$x_4$};
\fill [sblue] plot coordinates{(-1,-0.5)} node[below] {$s_{34}$};
\fill [sblue] plot coordinates{(0,1)} node[right] {$s_{24}$};
\fill [sblack] plot coordinates{(1,-0.5)} node[below](s14) {$s_{14}$};
\fill [sblack] plot coordinates{(1,0.5)} node[above right] {$s_{12}$};
\fill [sblue] plot coordinates{(0,-1)} node[above right] {$s_{13}$};
\fill [smarko] plot coordinates{(-0.37,0.45)} node[left] {$p$};
\fill [sgreen] plot coordinates{(-0.5,-0.75)} node[above right] {};
\fill [sgreen] plot coordinates{(1,0)} node[above right] {};
\end{scriptsize}
\end{tikzpicture}
\end{center}
\vspace{-0.5cm} \captionof{figure}{Case (7c)}
\end{minipage}
\noindent\begin{minipage}{0.45\textwidth}
\begin{center}
\begin{tikzpicture}[line cap=round,line join=round,x=0.9cm,y=0.9cm]
\clip(-2.5,-1.13) rectangle (2.46,2.4);
\draw (2,-1)-- (0,2);
\draw (-2,-1)-- (2,-1);
\draw (0,2)-- (-2,-1);
\draw (0,0)-- (0,2);
\draw (0,0)-- (2,-1);
\draw (0,0)-- (-2,-1);
\draw (0,-1)--(1,-0.5) -- (1,0.5);
\begin{scriptsize}
\fill [sred] plot coordinates{(2,-1)} node[above right] {$x_1$};
\fill [sred] plot coordinates{(0,2)} node[above] {$x_2$};
\fill [sred] plot coordinates{(-2,-1)} node[above left] {$x_3$};
\fill [sred] plot coordinates{(0,0)} node[below] {$x_4$};
\fill [sblue] plot coordinates{(1,-0.5)} node[below](s14) {$s_{14}$};
\fill [sblue] plot coordinates{(1,0.5)} node[above right] {$s_{12}$};
\fill [sblue] plot coordinates{(0,-1)} node[above left] {$s_{13}$};
\end{scriptsize}
\end{tikzpicture}
\end{center}
\vspace{-0.5cm} \captionof{figure}{Case (7d)}
\end{minipage}
\noindent\begin{minipage}{0.45\textwidth}
\begin{center}
\begin{tikzpicture}[line cap=round,line join=round,x=0.9cm,y=0.9cm]
\clip(-2.5,-1.13) rectangle (2.46,2.4);
\draw (2,-1)-- (0,2);
\draw (-2,-1)-- (2,-1);
\draw (0,2)-- (-2,-1);
\draw (0,0)-- (0,2);
\draw (0,0)-- (2,-1);
\draw (0,0)-- (-2,-1);
\draw (1,-0.5) -- (1,0.5);
\draw (0,1) -- (-1,0.5);
\draw (-1,-0.5) -- (0,-1);
\begin{scriptsize}
\fill [sred] plot coordinates{(2,-1)} node[above right] {$x_1$};
\fill [sred] plot coordinates{(0,2)} node[above] {$x_2$};
\fill [sred] plot coordinates{(-2,-1)} node[above left] {$x_3$};
\fill [sred] plot coordinates{(0,0)} node[right] {$x_4$};
\fill [sblue] plot coordinates{(-1,-0.5)} node[below] {$s_{34}$};
\fill [sblue] plot coordinates{(0,-1)} node[above] {$s_{13}$};
\fill [sblack] plot coordinates{(1,-0.5)} node[below](s14) {$s_{14}$};
\fill [sblack] plot coordinates{(1,0.5)} node[above right] {$s_{12}$};
\fill [sgreen] plot coordinates{(0,1)} node[right] {$s_{24}$};
\fill [sgreen] plot coordinates{(-1,0.5)} node[above left] {$s_{23}$};
\fill [sgreen] plot coordinates{(-0.5,-0.75)} node[above right] {$p$};
\fill [sblue] plot coordinates{(1,0)} node[left] {$q$};
\fill [sblack] plot coordinates{(-0.5,0.75)} node[above] {$r$};

\end{scriptsize}
\end{tikzpicture}
\end{center}
\vspace{-0.5cm} \captionof{figure}{Case (7e)}
\end{minipage}
\end{center}
\qed

\bigskip
\noindent\textit{Proof of Theorem \ref{mc_3(4)}}\\
Let $X$ be a properly 4-colored configuration. Assume $|X|\geq 13$. The largest
color class $C$ contains at least 4 points, so there is a convex or concave
$C$-empty set of 4 points. It is necessarily blocked by $3$ colors, meaning
that the blocking configuration is equivalent to the one described in Lemma
\ref{convex4block} or Lemma \ref{concave4block}. Both of these configurations
are maximal, it follows that $|X|\leq 10$, a contradiction.

The two configurations below show properly 4-colored 12-point sets, and they
prove $mc_3(4)=12$.  \qed

\vspace{-0.5cm}
\begin{center}
\begin{tikzpicture}[line cap=round,line join=round,x=0.25cm,y=0.25cm]
\clip(-9,-4) rectangle (23,24);
\draw (4,8)-- (10,8)-- (7,2.8)-- (4,8);
\draw (7.93,9.61)-- (22.14,-1.4)-- (5.5,5.4)-- (7.93,9.61);
\draw  (8.5,5.4)-- (3.65,5.4)-- (6.07,23.21)-- (8.5,5.4);
\draw (-7.21,-3.01)-- (7,8)-- (9.43,3.8)-- (-7.21,-3.01);
\begin{scriptsize}
\fill [bigblack] plot coordinates{(4,8)} node[above left]{};
\fill [bigblack] plot coordinates{(10,8)} node[above right]{};
\fill [bigblack] plot coordinates{(7,2.8)} node[below]{};
\fill [biggreen] plot coordinates{(7,8)} node[above left]{};
\fill [bigblue] plot coordinates{(8.5,5.4)} node[right]{};
\fill [bigred] plot coordinates{(5.5,5.4)} node[below left]{};
\fill [biggreen] plot coordinates{(9.43,3.8)} node[below]{};
\fill [bigblue] plot coordinates{(3.65,5.4)}  node[above left]{};
\fill [bigred] plot coordinates{(7.93,9.61)} node[above right]{};
\fill [bigred] plot coordinates{(22.14,-1.4)} node[below left]{};
\fill [bigblue] plot coordinates{(6.07,23.21)} node[right]{};
\fill [biggreen] plot coordinates{(-7.21,-3.01)} node[above left]{};
\end{scriptsize}
\end{tikzpicture}
\hspace{-0.5cm}
\begin{tikzpicture}[line cap=round,line join=round,x=1.2cm,y=1.2cm]
\clip(-2.2,-2.2) rectangle (3.2,3.2);
\tikzstyle{every node}=[inner sep=0pt]
\begin{scriptsize}
\node [nred] (A) at (0,0) {};
\node [nblack] (B) at (1,0) {};
\node [nblue] (C) at (1,1) {};
\node [ngreen](D) at (0,1) {};
\node [nred] (A') at (2,0) {};
\node [nblack] (B') at (1,2) {};
\node [nblue] (C') at (-1,1) {};
\node [ngreen] (D') at (0,-1) {};
\node [nred] (A'') at (-2,2) {};
\node [nblack] (B'') at (-1,-2) {};
\node [nblue] (C'') at (3,-1) {};
\node [ngreen] (D'') at (2,3) {};
\end{scriptsize}
\draw (A)--(B)--(C)--(D)--(A);
\draw (A)--(D')--(B'');
\draw (B)--(A')--(C'');
\draw (C)--(B')--(D'');
\draw (D)--(C')--(A'');
\draw (B'')--(A)--(B');
\draw (C'')--(B)--(C');
\draw (D'')--(C)--(D');
\draw (A'')--(D)--(A');
\draw (A)--(C');
\draw (B)--(D');
\draw (C)--(A');
\draw (D)--(B');
\end{tikzpicture}
\end{center}

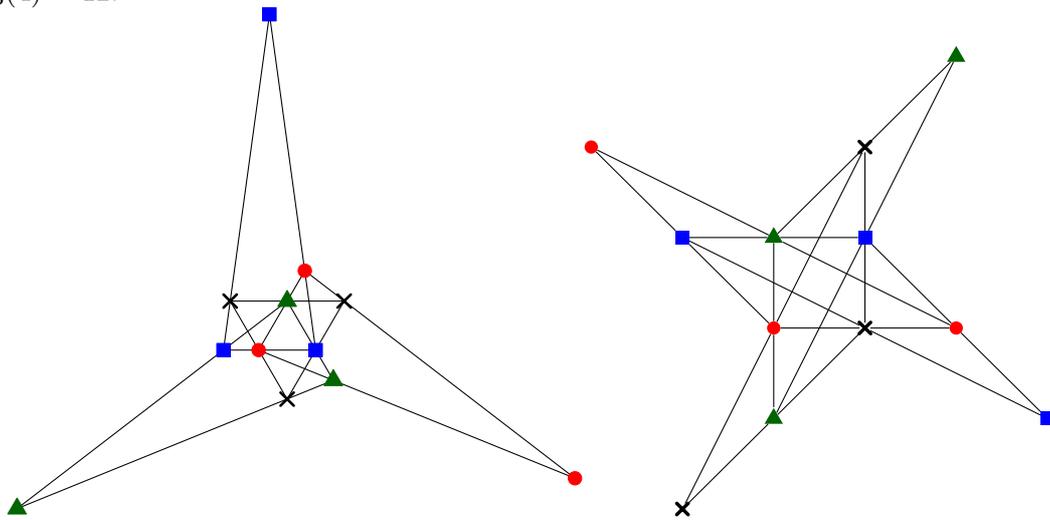
\end{document}